\documentclass[11pt]{article}
\UseRawInputEncoding
\usepackage{bbm}
\usepackage{amssymb, amsthm, amsmath, amscd}
\usepackage[usenames,dvipsnames]{color}
\newtheorem{thm}{Theorem}[section]
\newtheorem{lem}[thm]{Lemma}
\newtheorem{prop}[thm]{Proposition}
\newtheorem{cor}[thm]{Corollary}
\theoremstyle{definition}

\theoremstyle{definition}
\newtheorem{df}[thm]{Definition}
\theoremstyle{definition}
\newtheorem{rem}[thm]{Remark}

\theoremstyle{definition}
\newtheorem{exm}[thm]{Example}

\renewcommand{\phi}{\varphi}

\newcommand{\N}{\mathbb{N}}
\newcommand{\Z}{\mathbb{Z}}

\newcommand{\R}{\mathbb{R}}
\newcommand{\C}{\mathbb{C}}
\newcommand{\T}{\mathbb{T}}
\newcommand{\I}{\mathbb{I}}
\numberwithin{equation}{section}

\newcommand{\cpc}{c.p.c.~map}

\newcommand{\hm}{homomorphism}
\newcommand{\dt}{\delta}
\newcommand{\ep}{\varepsilon}

\newcommand{\la}{\langle}
\newcommand{\ra}{\rangle}
\newcommand{\andeqn}{\,\,\,{\rm and}\,\,\,}
\newcommand{\rforal}{\,\,\,{\rm for\,\,\,all}\,\,\,}
\newcommand{\CA}{$C^*$-algebra}
\newcommand{\SCA}{$C^*$-subalgebra}

\newcommand{\af}{{\alpha}}
\newcommand{\bt}{{\beta}}

\newcommand{\diag}{{\rm diag}}

\newcommand{\wilog}{without loss of generality}
\newcommand{\Wlog}{Without loss of generality}

\newcommand{\D}{\mathbb D}
\newcommand{\beq}{\begin{eqnarray}}
\newcommand{\eneq}{\end{eqnarray}}
\newcommand{\tforal}{\,\,\,\text{for\,\,\,all}\,\,\,}
\newcommand{\tand}{\,\,\,\text{and}\,\,\,}

\usepackage{amsfonts}
\usepackage{mathrsfs}
\usepackage{textcomp}
\usepackage[all]{xy}

\title{Almost representations}
\author{Huaxin Lin}
\date{
}

\begin{document}

\maketitle

\begin{abstract}
Let $H$ be an infinite dimensional separable Hilbert space, 
$B(H)$ the \CA\, of all bounded linear operators on $H,$
$U(B(H))$ the unitary group of $B(H)$ 
and ${\cal K}\subset B(H)$ the ideal of compact operators.  
Let $G$ be a countable discrete amenable group.
We prove the following:
For any $\ep>0,$ any finite subset ${\cal F}\subset G,$ and $0<\sigma\le 1,$
there exists  $\dt>0,$ finite subsets ${\cal G}\subset G$ and ${\cal S}\subset 
\C[G]$ satisfying the following property:
For any map $\phi: G\to U(B(H))$ such that
\beq\nonumber
\|\phi(fg)-\phi(f)\phi(g)\|<\dt\rforal f,g\in {\cal G}\tand\|\pi\circ \tilde \phi(x)\|\ge \sigma \|x\|
\tforal x\in {\cal S},
\eneq
there is a group \hm\, $h: G\to U(B(H))$ such that
\beq\nonumber
\|\phi(f)-h(f)\|<\ep\tforal f\in {\cal F},
\eneq
where $\tilde \phi$ is the linear extension of $\phi$ on the group 
ring $\C[G]$ and $\pi: B(H)\to B(H)/{\cal K}$ is the quotient map.
A counterexample is given that the fullness condition above cannot be removed.

We actually prove a more general result for separable amenable \CA s. 

\end{abstract}

\section{Introduction}

Let  $H$ be an infinite dimensional  separable Hilbert space and 
$B(H)$ the \CA\, of all bounded linear operators.  Consider a separable \CA\, $A$ and a  ($C^*$-) \hm\,
$h: A\to B(H),$  a representation of $A.$ 
Suppose that $L: A\to B(H)$ is a contractive completely positive linear map and almost multiplicative.
We are interested in the problem whether such a map $L$ is close to a genuine representation.
More precisely, we have  the following question:

{\bf Q1}:  Let $A$ be a separable \CA\, and $H$ be an infinite dimensional  separable 
Hilbert space.  Let ${\cal F}\subset A$ be a finite subset and $\ep>0.$
Are there a finite subset ${\cal G}\subset A$ and  a positive number $\dt>0$ satisfying the following:
for any contractive completely positive linear map $L: A\to B(H)$ with property that
\beq
\|L(ab)-L(a)L(b)\|<\dt\tforal a, b\in {\cal G},
\eneq
there is a \hm\,  $h: A\to B(H)$ such that
\beq
\|L(a)-h(a)\|<\ep\tforal a\in {\cal F}?
\eneq

In the special case that $A=C(\I^2),$ the \CA\, of continuous functions on the unit square,
and 
$H$ is any finite dimensional Hilbert space,
this is also known as von Neumann-Kadison-Halmos problem.
For this special case, it has an affirmative answer (see \cite{Halmos1}). 
However, prior to that, for the case that $A=C(\T^2)$ and 
$H$ is any finite dimensional Hilbert space, Dan Voiculescu 
gave a negative answer to the question (see \cite{DV2}). 
Almost multiplicative maps often appear in the study of \hm s
of \CA s, in particular, in the Elliott program of classification of amenable 
\CA s (see, for example, \cite[Theorem 5.1]{Linduke}, \cite[Theorem 8.7]{Linclr1}, 
\cite[Theorem 5.8]{LinAH1},
\cite[Theorem 12.7 ]{GLNI}, \cite[Theorem 5.4.6]{Linbook2}).  
There is also a significant development 
in the study of weak semiprojectivity (see, for example, \cite{Linsemiproj}, \cite{Sp}, \cite{ST}, \cite{EK}, etc.).
If $A$ is a weakly semiprojective, then the answer to {\bf Q1} is affirmative. 
For example, by \cite[Theorem 7.5]{Linsemiproj}, the answer to {\bf Q1}
is affirmative when $A$ is a separable purely infinite simple amenable 
\CA\, in the UCT class whose $K_i$-group ($i=0,1$) is a countable direct sum 
of finitely generated abelian groups. 

Recent interest in {\bf Q1} stems from the study of macroscopic observables and measurements. 
David Mumford  recently
asked whether an almost multiplicative map from a commutative \CA\,
to $B(H)$ can be approximated by a \hm\, (see \cite[Chapter 14]{Mf}).  We have  some affirmative 
solutions to {\bf Q1} in the case that $A$ is a commutative \CA\, with finitely many generators (see \cite{Linmf1} and \cite{Lincmp}).

The current study is also motivated by a 
question from Professor S. T. Yau during a brief SIMIS presentation.
Yau asked whether results in \cite{Linmf1} can be extended to 
nilpotent groups. 
 Consider a discrete countable amenable group 
$G$ and a ``quasi-representation", i.e, a map 
$\phi: G\to U(B(H)),$ the unitary group of $B(H),$ such 
that $\phi(fg)-\phi(f)\phi(g)$  has small norm on a finite subset of $G.$ 
The question is when there is a true \hm\, $h: G\to U(B(H))$ which is close to $\phi.$ 
More precisely,  one has the following question:

{\bf Q2}:  Let $G$ be a countable discrete amenable group 
and $H$ be an infinite dimensional separable Hilbert space.
Let ${\cal F}\subset G$ be a finite subset and 
$\ep>0.$ Are there a finite subset ${\cal G}\subset G$ and 
$\dt>0$ such that, for any map $\phi: G\to U(B(H))$ (unitary group of $B(H)$)
with
\beq\label{Q2-1}
\|\phi(fg)-\phi(f)\phi(g)\|<\dt\tforal f, g\in {\cal G},
\eneq
there is a group \hm\, $\psi: G\to U(B(H))$ 
such that
\beq\label{Q2-2}
\|\phi(f)-\psi(f)\|<\ep\rforal f\in {\cal F}?
\eneq

D. Kazhdan in \cite{Kz} proved the following theorem:
Let $G$ be an amenable group, $0<\ep<1/100$ and 
$\rho: G\to U(B(H))$ be a continuous map of $G$ 
such that 
$\|\rho(xy)-\rho(x)\rho(y)\|\le \ep$  for all $x, y\in G,$ then there is a \hm\, $h: G\to U(B(H))$
such that $\|\rho(g)-h(g)\|\le 2\ep$ for all $g\in G.$ 
Kazhdan's condition of ``almost multiplicative" is for all elements in the group uniformly.
In contrast, \eqref{Q2-1} imposes a weak local condition, and {\bf Q2} seeks a weaker approximation--- a natural fit for operator algebras. 
%
%
%
Unfortunately, the example in \ref{ExVoi} shows that the answer to {\bf Q2}, in general, has a negative 
answer, i.e., there are ``quasi-representations" which are far away from 
any representations.    A negative answer to {\bf Q2} also gives a negative answer 
to {\bf Q1}.  Nevertheless, we also provide a positive result for {\bf Q2} and Yau's question (see 
Theorem  \ref{Main2}) under an additional fullness condition.  Recently,  
Ruffs Willett had studied the same question as {\bf Q2} in the setting that the Hilbert space $H$ is of finite dimensional (see \cite{Wr}). 
So the results in this paper 
might be viewed as complements of Willett's results in the infinite dimension Hilbert spaces.

Suppose that $L: A\to B$ is a completely positive linear map, where 
$B$ is a unital \CA, and  $I\subset B$ is an ideal such that 
$L(A)\subset I+\C \cdot 1_B.$  Then, to understand $L,$ we may consider 
$L$ as a map from $A$ into $\tilde I$ (see Example \ref{ExVoi}). 
On the other hand, if $C$ is a \CA\, with an ideal $J$ such that 
$C/J\cong A$ and $\psi: C\to A$ is the quotient map. Then 
$L_1=L\circ \psi: C\to B$ is also a completely positive linear map.
However,  some information might be hidden (see Example \ref{Exshift}), if one insists to consider 
$L_1: C\to B$ instead of $L: A\to B.$  
These suggest that we should have a ``fullness" condition, 
for example, the second condition in \eqref{Main-1} (both maps in Section 8 are 
not full).

Denote by ${\cal N}$ the class of those separable amenable \CA s which satisfy the UCT.
Note that this class ${\cal N}$ contains all $AF$-algebras, all commutative \CA s, and their 
tensor products. 
It is closed under taking ideals  and quotients as well as 
inductive limits (see \cite{RS}). 

The first result of this paper can be  stated as follows:

\begin{thm}\label{MT-1}
Let $A$ be a separable amenable \CA\, in ${\cal N}.$
For any $\ep>0,$ any finite subset ${\cal F}\subset A,$ and $0<\lambda\le 1,$ 
there exists  $\dt>0$ and a finite subset ${\cal G}\subset A$ satisfying the following:

For any contractive  positive linear map  $L: A\to B(H)$ for some infinite dimensional separable Hilbert space $H$
which is ${\cal G}$-$\dt$-multiplicative, i.e.,
$
\|L(a)L(b)-L(ab)\|<\dt
$
for all $a, b\in {\cal G},$  such that
\beq\label{MT-1-00}
\|L(a)\|\ge \lambda \|a\| \tforal a \in {\cal G}
\eneq
and there is a separable \SCA\, $C\subset B(H)$
such that $L({\cal G})\subset C$ and $C\cap {\cal K}=\{0\},$  then
there is a \hm\,  $h: A\to B(H)$
such that
\beq
\|L(a)-h(a)\|<\ep\tforal a\in {\cal F}.
\eneq
\end{thm}

For purely infinite simple amenable \CA s, we have  the following result:

\begin{thm}\label{Tpurely}
Let $A$ be a separable amenable  purely infinite simple \CA\, in ${\cal N}.$
For any $\ep>0,$ any finite subset ${\cal F}\subset A,$ and $0<\sigma\le 1,$ 
there exists  $\dt>0$ and a finite subset ${\cal G}\subset A$ satisfying the following:

For any positive linear map $L: A\to B(H)$ for some infinite dimensional separable Hilbert space $H$
such that $1\ge \|L\|\ge \sigma$ and $
\|L(a)L(b)-L(ab)\|<\dt
$
for all $a, b\in {\cal G},$    then
there is a \hm\,  $h: A\to B(H)$
such that
\beq
\|L(a)-h(a)\|<\ep\tforal a\in {\cal F}.
\eneq
\end{thm}
One should note that, in general, purely infinite simple \CA s are not weakly semiprojective.
Moreover, $1\ge \|L \| \ge \sigma$  is not much different from  
$L\not=0$ and much weaker than 
 \eqref{MT-1-00}. In fact, when $\|L\|<\ep,$ we may simply choose $h=0.$

The main theorem of the paper is stated as follows. 


\begin{thm}\label{Main}
Let $A$ be a separable  quasidiagonal \CA\, in ${\cal N},$ $H$ be an infinite dimensional separable Hilbert space 
and $B(H)$ the \CA\, of all bounded linear operators on $H.$ 

For any $\ep>0,$ any finite subset ${\cal F},$ 
and $0<\sigma\le 1,$  there are $\dt>0$ and finite subsets ${\cal G}, {\cal H}\subset A$ 
satisfying the following:
For any contractive positive linear  map $L: A\to B(H)$ such that
\beq\label{Main-1}
\|L(ab)-L(a)L(b)\|<\dt\tforal a, b\in {\cal G}\tand 
\|\pi\circ L(c)\|\ge \sigma\|c\|\tforal c\in {\cal H},
\eneq
where $\pi: B(H)\to B(H)/{\cal K}$ is the quotient map,
there is a faithful  and full representation $h: A\to B(H)$ such that
\beq
\|L(a)-h(a)\|<\ep\tforal a\in {\cal F}.
\eneq
\end{thm}
A simplified version of the above is to choose ${\cal H}={\cal G}.$
Let us point out that
the fullness condition in the second part of \eqref{Main-1} is 
what one expected and much weaker 
than the ones in Theorem \ref{MT-1}.

As a corollary,  we have the following:

\begin{cor}\label{CM0}
Let $A$ be a separable simple quasidiagonal \CA\, in ${\cal N},$ $H$ be an infinite dimensional separable Hilbert space. 

For any $\ep>0,$ any finite subset ${\cal F},$ 
and $0<\sigma\le 1,$  there are $\dt>0$ and  a finite subset ${\cal G}\subset A$ 
satisfying the following:
For any  contractive positive linear map  $L: A\to B(H)$ such that 
$\|\pi\circ L\|\ge \sigma$ and 
\beq\label{CM-1}
\|L(ab)-L(a)L(b)\|<\dt\tforal a, b\in {\cal G}, 
\eneq
there is a faithful representation $h: A\to B(H)$ such that
\beq
\|L(a)-h(a)\|<\ep\tforal a\in {\cal F}.
\eneq
\end{cor}

In 
Corollary \ref{CM0}, 
the condition that $\|\pi\circ L\|\ge \sigma$ is actually necessary 
when $A$ is, in addition, unital and non-elementary (see \ref{Remarks6}).


For discrete amenable groups, 
we also offer the following:

\begin{thm}\label{Main2}
Let $G$ be a countable discrete amenable group and $H$ be an infinite dimensional Hilbert space.
Let $\ep>0$ and ${\cal F}\subset G$ be a finite subset and $1\ge \sigma>0.$
Then there exists $\dt>0,$ a finite subset ${\cal G}\subset G$ and 
a finite subset ${\cal S}\subset \C[G]$ such that, if 
$\phi: G\to U(B(H))$ is a map 
satisfying the condition that
\beq\label{Main2-00}
&&\|\phi(fg)-\phi(f)\phi(g)\|<\dt\tforal f,g\in {\cal G}\tand\\\label{MM-2}
&&\|\pi\circ \tilde \phi(a)\|\ge \sigma\|a\|\tforal a\in {\cal S}
\eneq
($\tilde \phi$ is the linear extension of $\phi$---see Definition \ref{dtildephi}),
then there exists a \hm\, $h: G\to U(B(H))$ such that
\beq
\|\phi(f)-h(f)\|<\ep\tforal f\in {\cal F}.
\eneq
Moreover, $h$ extends a faithful and full representation of $C^*_r(G).$
\end{thm}

Example \ref{ExVoi}  shows that the fullness condition
\eqref{MM-2} in Theorem \ref{Main2} cannot be removed even in the case 
that $G=\Z^2.$ 

This paper is organized as follows:
Section 2 is a preliminary.  In section 3, we provide the proof of Theorem \ref{MT-1} and 
Theorem \ref{Tpurely}. Section 4 is a discussion of Properties P1, P2, and P3.
In section 5, we present some absorbing results.  Section 6 is devoted to the proof of 
the main result, Theorem \ref{Main} and its corollary.  
In section 7, we prove Theorem \ref{Main2}.   In the last section, section 8, 
we first present a simple example which shows that, 
for Theorem \ref{Main}, the second condition in \eqref{Main-1} cannot be removed.
We also  present an 
 example of Voiculescu, which shows that the answer to {\bf Q2}
is negative without the fullness condition \eqref{MM-2} even in the case that $G=\Z^2.$  

\section{Fullness and Regularity}

\begin{df}\label{conventions}
All ideals in \CA s in this paper are closed, two-sided ideals. 
If $A$ is a  unital \CA, then $U(A)$ is the unitary group of $A.$
If $A$ is not unital, denote by $\tilde A$ the unitization of $A.$
Denote by $A^{\bf 1}$ the (closed) unit ball of $A$ and $A^{\bf 1}_+=A^{\bf 1}\cap A_+.$ 
\end{df}

\begin{df}
Let $\{B_n\}$ be a sequence of \CA s.
Denote by $l^\infty(\{B_n\})$ the $C^*$-product of $\{B_n\},$ 
i.e.
$$l^\infty(\{B_n\})=\{\{b_n\}: b_n\in B_n\andeqn  \sup_n\|b_n\|<\infty\}.
$$
Denote by $c_0(\{B_n\})$ the $C^*$-direct sum of $\{B_n\},$
i.e., 
$$c_0(\{B_n\})=\{\{b_n\}\in l^\infty(\{B_n\}): \lim_{n\to\infty}\|b_n\|=0\}.$$

Note that $c_0(\{B_n\})$ is an ideal of $l^\infty(\{B_n\}).$  If $B_n=B$ for all $n\in \N,$
we may write $l^\infty(B)$ and $c_0(B)$ instead. 
\end{df}

\begin{df}\label{Dvarpi}
We will fix a free  ultrafilter $\varpi$ of subsets of $\N$ which may be viewed as an element in $\bt(\N)\setminus \N$
(where $\bt(\N)$ is the Stone-{\v C}ech compactification of $\N$).
Write 
\beq
c_{0, \varpi}(\{B_n\})=\{\{b_n\}\in l^\infty(\{B_n\}): 
\lim_{n\to \varpi} \|b_n\|=0\}.
\eneq
Define $q_\varpi(\{B_n\}):=l^\infty(\{B_n\})/c_{0, \varpi}(\{B_n\}).$
Denote by $\pi_\varpi: l^\infty(\{B_n\})\to q_\varpi(\{B_n\})$ the quotient map. 
\end{df}

\begin{df}
Let $B$ be a \CA, $a, b\in B$ and let $\ep>0.$ 
We write 
\beq
a\approx_\ep b, 
\eneq
if $\|a-b\|<\ep.$   

Let $A$ be another \CA\, and $L_1, L_2: A\to B$ be two maps
and let ${\cal F}\subset B$ be a subset.
We write 
\beq
L_1(a)\approx_\ep L_2(a)\,\, {\rm on}\,\,\, {\cal F},
\eneq
if $L_1(a)\approx_\ep L_2(a)$ for all $a\in {\cal F}.$ 
\end{df}

\begin{df}\label{Dffdt}
Fix $\dt>0.$ Define 
$
f_\dt\in C(\R_+)
$
by $f_\dt(t)=0$ if $t\in [0, \dt/2],$ $f_\dt(t)=1$ if $t\in [\dt, \infty)$ and linear 
in $(\dt/2, \dt).$
\end{df}

\begin{df}
Let $A$ and $B$ be \CA s, and $L: A\to B$ be a linear map.
If $L$ is a completely positive contraction, we may write 
that $L$ is a \cpc. 
\end{df}

\begin{df}
Denote by ${\cal N}$ the class of separable amenable \CA s which satisfy the UCT
(see \cite{RS}).
\end{df}

Fix a separable amenable \CA\, $A$ and a positive number $M>0.$ 
Let $L_n: A\to B_n$ (any \CA\, $B_n$) be a sequence of positive linear maps 
such that  $\|L_n\|\le M$ and 
\beq
\lim_{n\to\infty}\|L_n(ab)-L_n(a)L_n(b)\|=0\rforal a, b\in A.
\eneq 
Define $\Lambda: A\to l^\infty(\{B_n\})$ by $\Lambda(a)=\{L_n(a)\}$ 
and $\psi: A\to l^\infty(\{B_n\})/c_0(\{B_n\})$  by $\psi=\Pi\circ \Lambda,$
where $\Pi:  l^\infty(\{B_n\})\to l^\infty(\{B_n\})/c_0(\{B_n\})$ is the quotient map.
Then $\psi$ is a \hm. 

By  applying  the  Choi-Effros Lifting Theorem (\cite{CE}), one obtains the following
proposition which will be used often in this paper.

\begin{prop}\label{PC-E}
Let $A$ be a separable amenable \CA.

For any $\ep>0$ and any finite subset ${\cal F},$ 
there are $\dt>0$ and a finite subset ${\cal G}\subset A$ satisfying the following:
For any positive linear  map $L: A\to B$ (for any \CA\, $B$) such that
$\|L\|\le 1$ and
\beq
\|L(ab)-L(a)L(b)\|<\dt \tforal a,\,b \in A,
\eneq
there is a completely positive linear map  $\phi: A\to B$ such that
\beq
\|L(a)-\phi(a)\|<\ep\tforal a\in {\cal F}.
\eneq
Moreover,  if $0<\sigma\le 1$ is given and one assumes that $\|L(a)\|\ge \sigma\|a\|$
for all $a\in {\cal F},$ we may choose $h$ such that
$\|h(a)\|\ge (\sigma/2)\|a\|$ for all $a\in {\cal F}.$
\end{prop}

\begin{df}\label{Dlocfull}
Let $A$ and $B$ be \CA s.
An element $b\in B$ is {\it full}, if any (closed) ideal containing $b$ is $B.$ 
Suppose that $L: A\to B$ is a positive linear map.
We say $L$ is {\it full}, if 
$L(a)$  is full 
for all $a\in A_+\setminus \{0\}.$

Let ${\cal F}\subset A_+\setminus \{0\}$ be a subset of $A.$ We say $L$ is ${\cal F}$-full, 
if  
$L(a)$  is full 
 for all $a\in {\cal F}.$

Let $N: A_+\setminus \{0\}\to \N$ and $K: A_+\setminus \{0\}\to \R_+$
be two maps. Write $T=(N, K): A_+\setminus \{0\}\to (\N, \R_+).$

Suppose that $L: A\to B$ is a positive linear map and $H\subset A_+\setminus \{0\}$ is a subset.
We say that $L$ is (uniformly) $T$-$H$-full, if, for any $a\in H,$  any
$b\in B^{\bf 1}_+,$ and $\ep>0,$ there exist
$x_1, x_2,...,x_{N(a)}$ with $\|x_i\|\le K(a),$ $i=1,2,...,N(a),$ such that 
\beq
\|\sum_{i=1}^{N(a)} x_i^* L(a)x_i-b\|<\ep
\eneq
Often, later, we require that, if $a\in H,$ then $f_{\|a\|/2}(a)\in H.$
\end{df}

\begin{rem}\label{Rfullid}
In Definition \ref{Dlocfull}, suppose that $B$ is unital, 
then, for any $a\in H$ and $\ep>0,$ there exist $x_1, x_2,...,x_{N(a)}$ with $\|x_i\|\le K(a),$ $i=1,2,...,N(a),$ such that 
\beq
\|\sum_{i=1}^{N(a)} x_i^* L(a)x_i-1_B\|<\ep.
\eneq
Choosing $0<\ep<1,$ then there is $c\in B_+$ with $\|c\|<{1\over{1-\ep}}$ 
such that
\beq
\sum_{i=1}^{N(a)} cx_i^*L(a)x_ic=1_B.
\eneq
Put $y_i=x_ic,$ $i=1,2,...,N(a).$ Then  $\|y_i\|\le K(a)(1/(1-\ep)),$ $i=1,2,...,N(a).$

Conversely, 
suppose that 
there are $y_1, y_2,...,y_{N(a)}\in B$ with $\|y_i\|\le K(a),$ $i=1,2,...,N(a),$
such that
\beq
\sum_{i=1}^{N(a)}y_i^*L(a)y_i=1_B.
\eneq
Then, for any $b\in B^{\bf 1}_+,$
\beq
\sum_{i=1}^{N(a)}b^{1/2}y_i^* L(a)y_i b^{1/2}=b.
\eneq
\end{rem}

\begin{df}
Let $H$ be an infinite dimensional Hilbert space.
Denote by $B(H)$ the \CA\, of all bounded linear operators on $H$ and 
${\cal K}$ the \CA s of all compact operators on $H.$ 
Denote by $\pi: B(H)\to B(H)/{\cal K}$ the quotient map.
\end{df}

%
%
%
%

\begin{prop}\label{Pfull}
Let $A$ be a \CA\, and $L: A\to B(H),$ where $H$ is an infinite dimensional separable Hilbert space,
be a positive linear map.  Suppose that 
${\cal G}\subset A_+\setminus \{0\}$ is a finite subset such that 
$\|\pi\circ L(g)\|\ge \lambda_g>0$ for all $g\in {\cal G},$ where $\pi: B(H)\to B(H)/{\cal K}$ is the quotient map.
Define $N(a)=1$  and $K(a)=\sqrt{r/\lambda}$ for all 
$a\in A_+\setminus \{0\},$  where $\lambda=\min\{\lambda_g: g\in {\cal G}\}.$
Then $L$ is $(N(g), K(g))$-${\cal G}$-full.
\end{prop}

\begin{proof}
Fix $r>1.$ 
Since 
$\|\pi\circ L(a)\|\ge \lambda_a$ for all $a\in {\cal G},$ 
by the spectral theorem, we have that  $L(a)\ge (\lambda_a/r) p$ for some projection $p\in B(H)\setminus {\cal K}.$
There is a partial isometry $v\in B(H)$ such that
\beq
v^*pv=1_{B(H)}.
\eneq
Set $x(a)=(\sqrt{r/\lambda_a}) v$ for $a\in {\cal G}.$ 
Then $\|x(a)\|= \sqrt{r/\lambda_a}$ for all $a\in {\cal G}.$
We have 
\beq
x(a)^*L(a)x(a)\ge 1_{B(H)}.
\eneq
Hence $L$ is $(N(a), K(a))$-${\cal G}$ -full.
\end{proof}

%
%
%

\begin{df} \label{Dregularity}(\cite[Definition 2.1]{GL-KT})
Let $B$ be a unital \CA. Denote by $U_0(B)$ the path connected component of $U(B)$ containing 
$1_B.$

Fix a map $r_0 : \N \to \Z_+,$  a map $r_1 : \N\to  \Z_+,$ a map $l : \N \times  \N \to \N,$ and 
integers
$s \ge 1$ and $R \ge 1.$  
We shall say a \CA\,   $A$  belongs to the class ${\bf C}_{(r_0,r_1,l ,s,R)},$ if

(a) for any integer $n\ge 1$ and any pair of projections $p, q \in M_n(\tilde A)$ with 
$[p] = [q] \in K_0(A), p \oplus 1_{M_{r_0(n)}}(\tilde A)$ and $q\oplus 1_{M_{r_0(n)}}(\tilde A)$
are Murray–von Neumann equivalent, and moreover, if $p\in M_n(\tilde A)$ and $
q \in M_m(\tilde A)$ and $[p] -[q] \ge 0,$ then there exists
$p'\in  M_{n+r_0(n)}(\tilde A)$ such that $p'\le p \oplus 1_{M_{r_0(n)}}$ and $p'$  is equivalent to 
$q \oplus 1_{M_{r_0(n)}};$

(b) if $k\ge 1,$  and $x \in K_0(A)$ such that $-n[1_{\tilde A}] \le kx \le n[1_{\tilde A}]$
for some integer $n \ge 1,$  then
$-l (n, k)[1_{\tilde A}] \le x \le l (n, k)[1_{\tilde A}];$

(c) the canonical map $U(M_s(\tilde A))/U_0(M_s(\tilde A))\to  K_1(\tilde A)$ is surjective;

(d) if $u \in U(M_n(\tilde A))$ and $[u] = 0$ in $K_1(\tilde A),$ then $u\oplus 1_{M_{r_1(n)}}\in
U_0(M_{n+r_1(n)}(\tilde A));$

(f) ${\rm cer}(M_m(\tilde A)) \le  R$  for all $m \ge 1.$ 

\vspace{0.2in}

If A has stable rank one, and (a), (c), and  (d) hold; and  they hold with $r_0 = r_1 = 0,$ and 
$s=1.$

Let $L\ge \pi$ and $A$ be a \CA.  Let us consider the condition

(f') ${\rm cel}(M_m(\tilde A))\le L$ for all $m\ge  1.$ 

This means that every unitary $u\in U_0(M_m(\tilde A))$  has a continuous path 
$\{u(t): t\in [0,1]\}$ such that $u(0)=u,$ $u(1)=1$ and 
the length of $\{u(t): t\in [0,1]\}$ is no more than $L.$   It is easy to see that
if $A$ satisfies condition (f'), then 
\beq
{\rm cer}(M_m(\tilde A))\le [L/2\pi]+1.
\eneq

Let $r:=r_0=r_1.$ 
Denote by ${\bf A}_{(r, l, s, L)}$ the class of \CA s which satisfy 
condition (a), (b), (c), (d) and (f') for $r_0=r_1=r,$ $l,$ $s$ and $L$ above.


\end{df}

\begin{rem}\label{Rregular}
If $H$ is an infinite dimensional separable Hilbert space, then 
$B(H)\in {\bf A}_{(0, 1, 1, 2\pi)}.$
If $B$ is a purely infinite simple \CA, 
then 
${\rm cel}(B)\le 2\pi$ (see \cite{Phpi}). 
It follows that $B\in {\bf A}_{(0, 1, 2, 2\pi)}.$
\end{rem}

\section{Almost representations}

\begin{df}\label{Dunital}
Let $A$ and $B$ be \CA s and $L: A\to B$ be a \cpc.
Suppose that ${\cal G}\subset A$ is a finite subset and $\dt>0.$
We say $L$ is ${\cal G}$-$\dt$-multiplicative, if 
\beq
\|L(ab)-L(a)L(b)\|<\dt \rforal a, b\in {\cal G}.
\eneq

If, in addition, $A$ is unital, in this paper, when we 
say $L$ is   ${\cal G}$-$\dt$-multiplicative, we always assume that 
$1_A\in {\cal G}.$   Of course, we consider only the case that ${\cal G}$ is large and $\dt$ is small.
Let $M=\max\{\|a\|: a\in {\cal G}\}$ (The most interesting case is the case that $M=1.$).
For example, we may always assume that $\dt<1/32(M+1).$ 
Hence  there is a projection $e\in B$ 
such that
\beq
\|L(1_A)-e\|<2\dt<1/16.
\eneq
Therefore
\beq
\|eL(1_A)e-e\|<2\dt.
\eneq
Choose $x\in eBe_+$ ($x=(eL(1_A)e)^{-1}$ in $eBe$) such that
\beq
xeL(1_A)e=eL(1_A)ex=e.
\eneq
Then 
\beq
\|x-e\|<{2\dt\over{1-2\dt}}.
\eneq
It follows that 
\beq
\|x^{1/2}-e\|\le {4\dt\over{1-2\dt}}
\eneq
Put $d=x^{1/2}$ and 
$\dt_0=(2\dt+{4\dt\over{1-2\dt}}).$  Then
\beq
\|d-L(1_A)\|<\dt_0
\eneq
Define  $L': A\to eBe$ by 
$L'(a)=dL(a)d$ for all $a\in A.$   
Then $L'(1_A)=dL(1_A)d=x^{1/2}eL(1_A)ex^{1/2}=eL(1_A)ex=e.$
Hence $L'$ is a \cpc. 
Moreover,
\beq
&&\hspace{-0.6in}L'(a)=dL(a)d\approx_{8\dt M\over{1-2\dt}} eL(a)e\approx_{2\dt M} L(1_A)L(a)L(1_A)\\
&&\approx_{2\dt} L(a)\rforal a\in {\cal G},\\
&&\hspace{-0.6in}L'(ab)=dL(ab)d\approx_{\dt} dL(a)L(b)d\approx_{M\dt}dL(a)L(1_A)L(b)d\\
&&\approx_{2M\dt_0}dL(a)ddL(b)d=L'(a)L'(b)\rforal a, b\in {\cal G}.\\
\eneq
Put $\dt'=(M+1)\dt+2M\dt_0>0.$ 
Then $L'$ is ${\cal G}$-$\dt'$-multiplicative and 
$L'(a)\approx_{\dt'} L(a)$ for all $a\in {\cal G}.$ 
In other words, we may consider only those \cpc s $L$ such that 
$L(1_A)$ is a projection.   This will be a convention of this paper. 
\end{df}

The following is a special case of \cite[Lemma 4.1.5]{Linbook2} and \cite[Theorem 3.14]{EGLN00}.

\begin{thm}[cf. Lemma 4.1.5 of \cite{Linbook2}, Theorem 3.14 of \cite{EGLN00} and Theorem 5.9 of \cite{LinAUCT}]\label{TAUCT}
Let $A$ be a separable amenable \CA\, satisfying the UCT,  and 
$r: \N\to \Z_+,$ $l: \N\times \N \to \N,$ $s\ge 1$ and $L\ge 2\pi.$ 
Then, for any finite subset ${\cal F}\subset A,$ $\ep>0$ and 
$T: A_+\setminus \{0\}\to (\N_+, \R_+),$ $a\mapsto (N(a), K(a)),$ there exists a finite subset ${\cal G}_1\subset A,$  a finite subset ${\cal G}_2\subset A_+\setminus \{0\},$ a positive number $\dt>0,$ a finite subset ${\cal P}\subset \underline{K}(A)$ and 
an integer $k\in \N$ (they do not depend on $T$ but 
on $A,$ $\ep$ and ${\cal F}$) such that, for any unital  \CA\, $B\subset {\bf A}_{(r, l, s, L)}$ and 
any ${\cal G}$-$\dt$-multiplicative \cpc s $\phi, \psi: A\to B$ and 
any 
${\cal G}$-$\dt$-multiplicative \cpc\, $\sigma: A\to B$ such that 
$\sigma$ is ${\cal G}_2$-full, 
if
\beq\label{KKcond}
{[}\phi{]}|_{\cal P}={[}\psi{]}|_{\cal P},
\eneq
and, if both $\phi(1_A)$ and $\psi(1_A)$ are invertible , or both are non-invertible, 
in the case that $A$ is unital,  
then there exists a unitary $U\in M_{k+1}(B)$ such that
\beq\label{TAUCT-2}
\|U^*(\diag(\phi(a), \overbrace{\sigma(a), \sigma(a),...,\sigma(a)}^k)U-\diag(\psi(a), 
\overbrace{\sigma(a), \sigma(a),...,\sigma(a)}^k)\|<\ep
\eneq
for all $a\in {\cal F}.$
\end{thm}

\begin{proof}
The  case that $A$ is unital follows from \cite[Lemma 4.15]{Linbook2} and 
the non-unital case follows from \cite[Theorem 3.14]{EGLN00}.
Note that, if $u\in M_N(\tilde A)$ is a unitary, 
then, by (f'), 
\beq
{\rm cel}(\la \phi(u)\ra \la \psi(u^*)\ra)\le 2L
\eneq
(whenever ${\cal G}$ is sufficiently large and $\dt$ is small). 
%
\end{proof}

\begin{rem}
If $B$ is in ${\bf C}_{(r_0,r_1,l ,s,R)},$  then the statement has to be altered a little bit.
After `` $T$," we will add a map ``${\bf L}: U(M_\infty(\tilde A))\to \R_+$".
Then, after  ``a finite subset ${\cal P}\subset \underline{K}(A)$", we will add
``a finite subset ${\cal U}\subset U(M_\infty(\tilde A))"$ and, then, together \eqref{KKcond},
we require  `` ${\rm cel}(\la \phi(u)\ra \la \psi(u^*)\ra)\le 2 {\bf L}$
for all $u\in {\cal U}$" .

\end{rem}

\begin{lem}\label{Luniq}
Let $A$ be a    separable  \CA\, in ${\cal N}.$ 
 For any $\ep>0,$  any   finite subset ${\cal F}\subset A$  and 
$0<\lambda\le 1,$ 
  there exists $\dt>0,$  finite subsets ${\cal G}_1\subset A,$ 
  ${\cal G}_2\subset A_+\setminus \{0\}$  and an integer $k\in \N$ satisfying the following:
  For any \cpc\, $L: A\to B(H),$ where $H$ is an infinite dimensional separable Hilbert space,
  such that  
  \beq
  \|L(g_1g_2)-L(g_1)L(g_2)\|<\dt\tforal g_1, g_2\in {\cal G}_1,
  \eneq
 there exists
 a unitary $u\in M_{k+1}(B(H))$ 
 such that
 \beq
 \|L(f)\oplus \diag(\overbrace{h_0(f), h_0(f), ...,h_0(f)}^k)-u^*(h(f)\oplus \overbrace{h_0(f), h_0(f), ...,h_0(f)}^k)u\|<\ep
 \eneq
 for all $f\in {\cal F},$
 any ${\cal G}_1$-$\dt$-multiplicative   \cpc\, 
$h_0: A\to B(H)$  such that $\|\pi\circ h_0(g)\|\ge \lambda\|g\|$
for all $g\in {\cal G}_2,$  
and any \hm\, $h: A\to B(H)$ (we also assume that, in the case that $A$ is unital,
if $L(1_A)$ is invertible, $h$ is unital, or $L(1_A)$ is not invertible, $h$ is not unital).
\end{lem}

\begin{proof}
We will apply Theorem \ref{TAUCT}.

We choose $r=0,$ $l=1,$ $s=1$ and $L=2\pi.$   Then $B(H)\subset {\bf A}_{(0,1,1,2\pi)}$
(see \ref{Rregular}). 

Define $T: A\to (\N_+, \R_+),$ $a\mapsto (1, (2/\lambda) \|a\|).$ 
Let ${\cal G},$ $\dt>0$ and ${\cal P}\subset \underline{K}(A)$ and $k\in \N$  be given  by \ref{TAUCT}
for ${\cal F}$, $\ep,$ 
and $T$ (as well as $r=0,$ $l=1,$ $s=1$ and $L=2\pi$). 

Since $K_i(B(H))=\{0\},$ $i=0,1,$ 
for any ${\cal G}$-$\dt$-multiplicative, $[L]|_{\cal P}=0$ and 
$[h]|_{\cal P}=0.$ 

 Suppose that $h_0: A\to B(H)$ is a ${\cal G}_1$-$\dt$-multiplicative \cpc\, such that
 $\|\pi\circ h_0(g)\|\ge \lambda\|g\|$ for all $g\in {\cal G}_2.$
 Then, by Proposition \ref{Pfull}, $h_0$ is also $T$-${\cal G}_2$-full.
 Hence Theorem \ref{TAUCT} applies.
 %

\end{proof}

\begin{lem}\label{Luniqh}
Let $A$ be a separable  \CA\, in ${\cal N}.$ 
 For any $\ep>0$ and any finite subset ${\cal F}\subset A,$  
  there exists $\dt>0$ and  a finite subset ${\cal G}\subset A$ 
  satisfying the following:
  For any \cpc\, $L: A\to B(H),$ where $H$ is an infinite dimensional separable Hilbert space,
  such that  
  \beq
  \|L(g_1g_2)-L(g_1)L(g_2)\|<\dt\tforal g_1, g_2\in {\cal G},
  \eneq
 there exists
 a unitary $u\in M_2(B(H))$ 
 such that
 \beq
 \|L(f)\oplus h_0(f)
 -u^*(h(f)\oplus 
  h_0(f))u\|<\ep
 \eneq
 for all $f\in {\cal F},$ and for 
any full \hm\, 
$h_0: A\to B(H)$  
and any \hm\, $h: A\to B(H)$ (we also assume that, in the case that $A$ is unital,
if $L(1_A)$ is invertible, $h$ is unital, or $L(1_A)$ is not invertible, $h$ is not unital).
\end{lem}

\begin{proof}
Since $h_0$ is a full \hm, $\pi\circ h$ is also injective.  In particular, $\|\pi\circ h_0(a)\|=\|a\|$ for all $a\in A.$
Hence, by Proposition  \ref{Pfull}, $h_0$ is $(1, \|a\|)$-${\cal G}$-full.

Applying  Lemma \ref{Luniq},
we obtain an integer $k\in \N$ and a unitary $v\in M_{k+1}(B(H))$ such that 
\beq
L\oplus \overbrace{h_{0}\oplus h_{0}\oplus ...,\oplus h_{0}}^k\approx_{\ep/3} 
v^*(h\oplus \overbrace{h_{0}\oplus h_{0}\oplus ...\oplus h_{0}}^k)v \,\,\,{\rm on}\,\,\, {\cal F}.
\eneq
There is an isometry 
$w_0\in M_{k+1}(B(H))$ such that
\beq\label{xMT-1-0} 
w_0^*w_0=1\andeqn w_0w_0^*=1_{M_{k+1}(B(H))}.
\eneq 
Define $H_0: A\to B(H)$ by 
\beq\label{xMT-1-2}
H_0(a)=w_0^*\diag(\overbrace{h_0(a), h_0(a),...,h_0(a))}^{k})w_0\rforal a\in A.
\eneq

By Voiculescu's Weyl-von Neumann Theorem (\cite{DV1}), there is a unitary $w\in B(H)$ such that  
\beq
w^* H_0(a)w
\approx_{\ep/3} 
h_{0}(a)\tforal a\in {\cal F}.
\eneq
Put $w_1=1_{B(H)}\oplus w_0w.$ 
 Then 
 \beq
L(a)\oplus h_0(a)&\approx_{\ep/3}& 
 w_1^*(L(a)\oplus \overbrace{h_{0}(a)\oplus h_{0}(a)\oplus ...,\oplus h_{0}(a)}^k)w_1\\
&\approx_{\ep/3}& w_1^*v^*(h(a)\oplus \overbrace{h_{0}(a)\oplus h_{0}(a)\oplus ...\oplus h_{0}(a)}^k)vw_1\\
&\approx_{\ep/3} &w_1^*v^*w_1(h(a)\oplus h_0(a))w_1^*vw_1.
 \eneq
Put $u=w_1^*vw_1.$

\end{proof}

The following is the first result of this paper about almost multiplicative maps.

\begin{thm}\label{MT-1+}
Let $A$ be a separable amenable \CA\, satisfying the UCT.
For any $\ep>0,$ any finite subset ${\cal F}\subset A,$ and $0<\lambda\le 1,$ 
there exist  $\dt>0$ and a finite subset ${\cal G}\subset A$ satisfying the following:

For any contractive positive linear map $L: A\to B(H),$ where $H$ is an  infinite dimensional separable Hilbert space, 
which is ${\cal G}$-$\dt$-multiplicative, i.e.,
$
\|L(a)L(b)-L(ab)\|<\dt
$
for all $a, b\in {\cal G},$  such that
\beq
\|L(a)\|\ge \lambda \|a\| \tforal a \in {\cal G}
\eneq
and there is a separable \SCA\, $C\subset B(H)$
such that $L({\cal G})\subset C$ and $C\cap {\cal K}=\{0\},$  then
there is a \hm\,  $h: A\to B(H)$
such that
\beq
\|L(a)-h(a)\|<\ep\tforal a\in {\cal F}.
\eneq
\end{thm}

\begin{proof}
Fix $\ep>0,$ a finite subset ${\cal F}\subset A$  and $0<\lambda \le 1.$

Let $\dt_1>0$ (in place of $\dt$), ${\cal G}_1\subset A$ and ${\cal G}_2\subset A_+\setminus \{0\}$ and 
$k\in \N$ be integers given by Lemma \ref{Luniq} associated with $\ep/4,$
${\cal F}$ and $\lambda/2.$ We also assume that  $\dt_1$ (in place of $\dt$) and ${\cal G}_1$
(in placed of ${\cal G}$) work for Lemma \ref{Luniqh}. 

Put ${\cal G}_3={\cal G}_1\cup {\cal G}_2\cup {\cal F}.$ 
Let $\dt_2>0$ (in place of $\dt$) and ${\cal G}_4$ 
be required by Proposition \ref{PC-E} for 
$\min\{\ep, \dt_1\}$ (in place of $\ep$) and ${\cal G}_3$ (in place of ${\cal F}$).
Choose $\dt=\min\{\ep, \dt_1, \dt_2\}>0$ and 
${\cal G}={\cal G}_4\cup {\cal G}_3.$

Now suppose that $L: A\to B(H)$ satisfies the assumption 
for the above mentioned $\dt$ and ${\cal G}.$ By applying Proposition \ref{PC-E},    we may assume, \wilog, that 
$L$ is a \cpc.

There is an isometry 
$w\in M_{k+1}(B(H))$ such that
\beq\label{MT-1-0} 
w^*w=1\andeqn ww^*=1_{M_{k+1}(B(H))}.
\eneq 
Let $j: C\to B(H)$ be the embedding.
Define $j_{k+1}: C\to B(H)$ by 
\beq\label{MT-1-2}
j_{k+1}(c)=w^*\diag(\overbrace{j(c), j(c),...,j(c)}^{k+1})w\rforal c\in C.
\eneq
Choose ${\cal G}_5=\{L(a): a\in {\cal G}\}.$ 
By Voiculescu's Weyl-von Neumann Theorem (\cite{DV1}), 
there is a unitary $v\in B(H)$ such that
\beq\label{MT-1-3}
\|v^*j_{k+1}(c)v-j(c)\|<\ep/4\rforal c\in {\cal G}_5.
\eneq

Since $A$ is a separable amenable \CA, there is an embedding from $j_o: A\to O_2$
(see \cite{KP}).  Let $\iota_O: O_2\to B(H)$ be a unital embedding.
Define $\rho=\iota_O\circ j_o: A\to B(H).$
Note that $\pi\circ \rho$ is injective.

Note also  since $\pi |_C$ is injective, 
we also have 
\beq
\|\pi\circ j\circ L(a)\|\ge \lambda \|a\|\rforal a\in {\cal G}.
\eneq

If $A$ is unital, we may assume (with sufficiently small $\dt$ and large ${\cal G}$---see
Definition \ref{Dunital})
that $L(1_A)$ is a projection.  Since unital hereditary \SCA\, of $O_2$ is isomorphic to $O_2,$
when $L(1_A)$ is the identity of $B(H),$ we may choose $\rho$ to be unital. 
Otherwise, we may choose $\rho$ non-unital. 

By applying Lemma \ref{Luniq}, we obtain a unitary 
$v_0\in M_{k+1}(B(H))$ such that
\beq\label{MT-1-4}
\|L(a)\oplus {\rm diag}(\overbrace{L(a),L(a),...,L(a)}^k)-
v_0^*(\rho(a)\oplus {\rm diag}(\overbrace{L(a),L(a),...,L(a)}^k))v_0\|<\ep/4
\eneq
for all $a\in {\cal F}.$

There is also an isometry $v_1\in M_k(B(H))$ such that
\beq\label{MT-1-5}
v_1^*v_1=1_{B(H)}\andeqn v_1v_1^*=1_{M_k(B(H))}.
\eneq
Put $v_2=1_{B(H)}\oplus v_1\in M_{k+1}(B(H)).$  
Hence
\beq\label{MT-15+}
v_2^*v_2=1_{M_2(B(H))}\andeqn v_2v_2^*=1_{M_{k+1}(B(H))}.
\eneq
Define  $L_1: A\to B(H)$ by
\beq\label{MT-1-6}
L_1(a)=v_1^*{\rm diag}(\overbrace{L(a),L(a),...,L(a)}^k)v_1\rforal a\in A.
\eneq
Applying Lemma \ref{Luniqh}, we obtain a unitary $v_3\in M_{2}(B(H))$
such that
\beq\label{MT-1-7}
\|v_3^*(\rho(a)\oplus L_1(a))v_3-(\rho(a)\oplus \rho(a))\|<\ep/4
\eneq
for all $a\in {\cal F}.$ 
By Voiculescu's Weyl-von Neumann Theorem again, there is an isometry $v_4\in M_2(B(H))$
such that
\beq\label{MT-1-8}
&&v_4^*v_4=1_{B(H)},\,\, v_4v_4^*=1_{M_2(B(H))}\andeqn\\\label{MT-1-9}
&&\|v_4^*\diag(\rho(a), \rho(a))v_4-\rho(a)\|<\ep/4\rforal a\in {\cal F}.
\eneq

Now, by \eqref{MT-1-3}, \eqref{MT-1-4}, \eqref{MT-1-6},
\eqref{MT-1-7},  \eqref{MT-1-9}, we have
\beq
L(a)&\approx_{\ep/4}& v^*w^*\diag(\overbrace{L(a), L(a),...,L(a)}^{k+1})wv\\
&\approx_{\ep/4}&v^*w^*v_0^*(\rho(a)\oplus \diag(\overbrace{L(a),L(a),...,L(a)}^k)v_0wv\\
&=&v^*w^*v_0^*v_2(\rho(a)\oplus L_1(a))v_2^*v_0wv\\
&\approx_{\ep/4}& v^*w^*v_0^*v_2v_3(\rho(a)\oplus\rho(a))v_3^*v_2^*v_0wv\\\label{MT-1-8+}
&\approx_{\ep/4}& v^*w^*v_0^*v_2v_3v_4\rho(a)v_4^*v_3^*v_2^*v_0wv\,\,\,\,\rforal a\in {\cal F}.
\eneq
Put $z=v_4^*v_3^*v_2^*v_0wv.$ 
Note that $v_3\in M_2(B(H)),$ $v_0\in M_{k+1}(B(H))$ and $v\in B(H)$  are  unitaries.
Then by \eqref{MT-1-8}, \eqref{MT-15+} and \eqref{MT-1-0},
\beq
z^*z&=&v^*w^*v_0^*v_2v_3v_4v_4^*v_3^*v_2^*v_0wv=v^*w^*v_0^*v_2v_31_{M_2(B(H))} v_3^*v_2^*v_0wv\\
&=& v^*w^*v_0^*v_2v_2^*v_0wv=v^*w^*v_0^* 1_{M_{k+1}(B(H))}v_0wv\\
&=& v^*w^*1_{M_{k+1}(B(H))}wv=v^*1v=1.
\eneq
Similarly,
\beq
zz^*&=&v_4^*v_3^*v_2^*v_0wvv^*w^*v_0^*v_2v_3v_4=v_4^*v_3^*v_2^*v_0ww^*v_0^*v_2v_3v_4\\
&=&v_4^*v_3^*v_2^*v_01_{M_{k+1}(B(H))}v_0^*v_2v_3v_4\\
&=&v_4^*v_3^*v_2^*1_{M_{k+1}(B(H))}v_2v_3v_4\\
&=&v_4^*v_3^*1_{M_2(B(H))}v_3v_4\\
&=&v_4^*v_3^*1_{M_2(B(H))}v_3v_4=v_4^*1_{M_2(B(H)}v_4=1.
\eneq
Therefore $z\in B(H)$ is a unitary. Define $h: A\to B(H)$ by
\beq
h(a)=z^*\rho(a)z\rforal a\in A.
\eneq
It follows from \eqref {MT-1-8+} that
\beq
\|L(a)-h(a)\|<\ep\rforal a\in {\cal F}.
\eneq
\end{proof}

{\bf The proof of Theorem \ref{Tpurely}:}
Let $A$ be a separable amenable purely infinite simple \CA\, in ${\cal N}.$
Fix $\ep>0$ and a finite subset ${\cal F}\subset A.$ 

We first assume that $A$ is unital. 

Let $\dt_1$ (as $\dt$) and ${\cal G}_1\subset A$ (as ${\cal G}$) be given 
by Lemma  \ref{Luniqh}  for $\ep/2$ and ${\cal F}.$ 

Let ${\cal F}_1={\cal F}\cup {\cal G}_1$ and $\ep_1=\min\{\ep/4, \dt_1/2\}>0.$
We apply \cite[Lemma 7.2]{Linsemiproj}.
Let $\dt_2$ (as $\dt$) and  ${\cal G}\subset A$ be a finite subset given 
by \cite[Lemma 7.2]{Linsemiproj} for ${\cal F}_1$ (instead of ${\cal F}$)
and $\ep_1$ (instead of $\ep$).      

Now suppose that $L$ satisfies  the assumption of Theorem \ref{Tpurely}
for $\dt$ and ${\cal G}$ as above.  Applying Proposition \ref{PC-E}, 
by choosing possibly even smaller $\dt$ and larger ${\cal G},$ we may assume 
, \wilog, that $L$ is a \cpc.

As mentioned in \ref{Dunital}, we may assume 
that $L(1_A)=e\in B(H)$ is a projection. 
There is a unital embedding $\iota:O_2\to B(H)$ and 
a unital embedding $j: O_2\to eB(H)e.$  
Let $C$ be the separable \SCA\, of $B(H)$ generated by $L(A),$ $j(O_2)$ and  $\iota(O_2).$  Note that $C$ has a unital \SCA\, isomorphic to $O_2.$
Since $A$ is amenable, by \cite{KP}, there is a unital injective \hm\,
$h_0: A\to O_2.$

Applying \cite[Lemma 7.2]{Linsemiproj}, 
one obtains an isometry $v\in M_2(C)\subset M_2(B(H))$ 
such that
\beq
&&vv^*=1_{B(H)}\andeqn\\\label{Tpurely-5}
&&\|v^*(L(a)\oplus j\circ h_0(a))v-L(a)\|<\ep_1\rforal a\in {\cal F}_1.
\eneq
By Lemma \ref{Luniq}, there is a unitary $u\in M_2(B(H))$  and a \hm\, 
$h_1: A\to B(H)$ such that 
\beq\label{Tpurely-6}
\|u^*(h_1(a)\oplus j\circ h_0(a))u-L(a)\oplus j\circ h_0(a)\|<\ep/2\rforal a\in {\cal F}.
\eneq
Define $h: A\to B(H)$ by $h(a)=v^*u^*(h_1(a)\oplus j\circ h_0(a))uv$
for all $a\in A.$
Then, by \eqref{Tpurely-5} and \eqref{Tpurely-6},  for all $a\in {\cal F}.$
\beq
\hspace{-0.4in}\|L(a)-h(a)\|&\le &\|L(a)-v^*(L(a)\oplus j\circ h_0(a))v\|\\
&&+\|v^*(L(a)\oplus j\circ h_0(a))v-v^*u^*(h_1(a)\oplus j\circ h_0(a))uv\|\\
&<& \ep_1+\ep/2\le \ep.
\eneq

For the case that $A$ is not unital, then, since $A$ is purely infinite, by \cite{Zh}, 
it has real rank zero.  Therefore 
there is a projection $p\in A$ such that
\beq
\|pxp-x\|<\ep/4\rforal x\in {\cal F}.
\eneq
Thus, by replacing ${\cal F}$ by $\{pxp: x\in {\cal F}\}$ and 
$\ep$ by $\ep/2,$ we may reduce the general case to the case that 
$A$ is unital.
This completes the proof of Theorem \ref{CM0}.

\section{Property P1, P2 and P3}

\begin{df}\label{DP1} (see Section 2 of \cite{Linfull})

Let $B$ be a unital \CA. We say $B$ has property P1, if for every full element $b\in B$ there exist 
$x,\, y\in B$ such that $xby=1.$
If $b$ is positive, it is easy to see that $xby=1$ implies that there is $z\in B$ such that $z^*bz=1.$
\end{df}

\begin{df}\label{DP2}(see Section 2 of \cite{Linfull})

Let $B$ be a unital \CA. We say $B$ has property P2, if $1$ is properly infinite, that is, 
if there is a projection $p\not=1$ and partial isometries $w_1, w_2\in B$ such that
\beq
w_1^*w_1=1,\,\, w_1w_1^*=p,\,\, w_2^*w_2=1\andeqn w_2w_2^*\le 1-p.
\eneq
\end{df}

Every unital purely infinite simple \CA\, has properties P1 and P2.

\begin{df}\label{DP3}(see Section 2 of \cite{Linfull})

Let $B$ be a unital \CA. We say that $B$ has property P3, if for any separable \SCA\, $A\subset B,$
there exists a sequence of elements $\{a_n^{(i)}\}_{n\in \N}\in l^\infty(B),\,\,\, i=1,2,...,$ satisfying the following:

(a) $0\le a_n^{(i)}\le 1 \,\, \rforal i \andeqn n;$

(b) $\lim_{n\to\infty}\|a_n^{(i)}c-ca_n^{(i)}\|=0\rforal i\andeqn c\in A;$

(c) $\lim_{n\to\infty}\|a_n^{(i)}a_n^{(j)}\|=0,\,\,{\rm if}\,\,\, i\not=j;$

(d) $\{a_n^{(i)}\}_{n\in \N}$ is a full  element in $l^\infty(B)$ for all $i\in \N.$

Note that $\{a_n^{(i)}\}\not\in c_0(B)$ for any $i,$ since it is full in $l^\infty(B).$
In fact, 
\beq
\lim\inf_n\|a_n^{(i)}\|>0.
\eneq
Otherwise there would be a subsequence $\{n_k\}$ such that
\beq
\lim_k\|a_{n_k}^{(i)}\|=0.
\eneq
Note that $I=\{\{b_n\}\in l^\infty(B): \lim_k\|b_{n_k}^{(i)}\|=0\}$ is an ideal of $l^\infty(B)$ and 
$\{a_n^{(i)}\}\in I.$ This contradicts with (d). 
Hence, 
by replacing $a_k^{(i)}$ by $a_{k}^{(1)}/\|a_{k}^{(i)}\|,$ $k\in \N,$ 
we see $\{a_{k}^{(i)}\}$ satisfies (a), (b), and (d).
Thus, in what follows, we may 
assume, in the definition above,  that $\|a_n^{(i)}\|=1.$
\end{df}

The following is taken from \cite[Lemma 3.1]{Lincs2}. We present here since we use 
some details in the proof.

\begin{prop}\label{Pappid}
Let $A$ be a non-unital but $\sigma$-unital \CA. 
Suppose that $\{E_n\}$ is an approximate identity for 
$A$ such that
\beq
E_{n+1}E_n=E_n\tforal n\in\N.
\eneq
Then, for each separable \SCA\, $D\subset A,$ 
there exists an approximate identity $\{e_n\}\subset A$ 
such that
\beq
e_{n+1}e_n=e_n\tforal n\in \N\tand \lim_{n\to\infty}\|e_na-ae_n\|=0\tforal a\in D,
\eneq
where $e_n\subset {\rm Conv}(E_i: i\ge n).$
Moreover, we may assume that there are   subsequences
$k(n)<k(n+1),$ $n\in\N$ such that 
\beq
e_{n+1}\ge E_{k(n+1)},\,\, E_{k(n+1)}e_n= e_n\ge E_{k(n)}\rforal n\in \N.
\eneq
\end{prop}

\begin{proof}
Let $\{x_k\}\subset D$ be a dense sequence in the unit ball of $D.$
Then (see the proof of \cite[Theorem 3.12.14]{Pdbook})
there is a sequence of elements  $\{e_n'\}$
such that
\beq\label{Pappid-2}
\|e_n'x_k-x_ke_n'\|<1/n\rforal k\le n,
\eneq
where $e_n'$ is in the convex hull of 
$\{E_i: i\ge n\}.$  Hence 
\beq\label{Pappid-2+}
\lim_{n\to\infty}\|e_n'd-de_n'\|=0\rforal d\in D.
\eneq
Suppose that $e_n'=\sum_{i=n}^{m(n)} \af_{n,i}E_i\subset {\rm Conv}(E_i: i\ge n),$
where $\af_{n,i}\ge 0$ and $\sum_{i=n}^{m(n)} \af_{n,i}=1.$
Since $E_{i+1}E_i=E_i,$ 
we have 
\beq\nonumber
e_n'&=&(\sum_{i=n}^{m(n)}\af_{n,i}) E_n+(\sum_{i=n+1}^{m(n)}\af_{n,i})(E_{n+1}-E_n)+
(\sum_{i=n+2}^{m(n)}\af_{n,i}) (E_{n+2}-E_{n+1})+\\
&&\hspace{1in}\cdots + \af_{n,m(n)}(E_{m(n)}-E_{m(n)-1})\\
&=&E_n+(\sum_{i=n+1}^{m(n)}\af_{n,i})(E_{n+1}-E_n)+
(\sum_{i=n+2}^{m(n)}\af_{n,i}) (E_{n+2}-E_{n+1})+\\\label{Pappid-3}
&&\hspace{1in}\cdots + \af_{n,m(n)}(E_{m(n)}-E_{m(n)-1}).
\eneq
Hence
 \beq
 E_n\le e_n'\le E_{m(n)}.
 \eneq
 Moreover, if $m>m(n),$ 
 \beq\label{Pappid-4}
 e_n'E_m=e_n'.
 \eneq
 In particular,
 \beq\label{Pappid-5}
 e_n'E_{m(n)+1}=e_n'. 
 \eneq
 Choose $k(1)=1.$  $e_1=e_1',$ $k(2)=m(1)+1,$ 
$e_2=e_{m(1)+1}',$  $k(3)=m(m(1)+1)+1,$ and 
$e_3=e_{k(3)}',....$
Then, by \eqref{Pappid-3} and by induction,
we may construct $\{e_n\}\subset \{e_n'\}$ such
that
\beq\label{Pappid-8}
e_{n+1}\ge E_{k(n+1)},\,\, E_{k(n+1)}e_n=e_n\ge E_{k(n)},\,\,\, n\in \N.
\eneq
Moreover, we have, by \eqref{Pappid-5} and by \eqref{Pappid-3}
\beq
e_{n+1}e_n=e_n,\,\,n\in \N.
\eneq
Furthermore, by \eqref{Pappid-2+},
\beq
\lim_{n\to\infty}\|e_na-ae_n\|=0\rforal a\in D.
\eneq
It follows from \eqref{Pappid-8} that $\{e_n\}$ forms an approximate identity for $A.$
%
%
\end{proof}

\begin{lem}\label{Lold}
Let $A$ be a unital \CA\, and $B=A\otimes {\cal K}.$
Then $M(B)$  and $M(B)/B$ have property P3.
\end{lem}

\begin{proof}
Let $\pi: M(B)\to M(B)/B$ be the quotient map and $D$ a separable \SCA\, of $M(B).$
Let $\{d_k\}$ be a dense sequence in the unit ball of $D.$ 
Set ${\cal F}_n=\{d_1, d_2,...,d_n\},\,\, n\in \N.$ 
Let $\{e_{i,j}\}$ be a system of matrix units for ${\cal K}.$
Set $E_n=\sum_{i=1}^n e_{i,i},$ $n\in \N.$ 

Applying Proposition \ref{Pappid},  we choose a quasi-central approximate identity $\{e_n\}$
such that
\beq\label{Lold-2}
e_{n+1}e_n=e_n\andeqn \lim_{n\to\infty}\|e_na-ae_n\|=0\rforal a\in D.
\eneq
We may assume that
\beq\label{Lold-3}
\|e_nd-de_n\|<1/2^n\rforal d\in {\cal F}_n,\,\,\,n\in \N.
\eneq
Moreover, there is a  subsequence $\{k(n)\}$ of $\N$ such that
\beq\label{Lconscale-3}
e_{n+1}\ge E_{k(n+1)},\,\,\, E_{k(n+1)}e_n=e_n\ge E_{k(n)},\,\,n\in \N.
\eneq
Note that $k(n+1)>k(n).$
Define $e_1'=e_1,$ $e_n'=e_{2n+1},$ $n\in \N.$
Then 
\beq\label{Lconscale-4}
&&e_{n+1}'e_n'=e_n'\andeqn\\\label{Lconscale-5}
&&e_{n+1}'-e_n'=e_{2n+3}-e_{2n+1}\ge E_{k(2n+3)}-e_{2n+1}\ge E_{k(2n+3)}-E_{k(2n+2)}.
\eneq
Note  that, if $|n-m|\ge 2,$
\beq
(e_{n+1}'-e_n')(e_{m+1}'-e_m')=0.
\eneq
Choose an infinite subset $F\subset \N$ such that
if $n\not=m$ are in $F,$ then $|n-m|\ge 2.$ 
Define 
\beq
a_{F}=\sum_{n\in F}(e_{n+1}'-e_n')
\eneq
(which converges in the strict topology).
By \eqref{Lconscale-5}, 
\beq
a_{F}\ge \sum_{n\in F} (E_{k(2n+3)}-E_{k(2n+2)})
\eneq
and the right term converges in the strict topology.

In $M(A\otimes {\cal K}),$ there is a partial isometry $w_F$ such that
\beq
w_F^* \sum_{n\in F} (E_{k(2n+3)}-E_{k(2n+2)})w_F=1_{M(A\otimes {\cal K})}.
\eneq
Note that $\|w_F\|=1.$ 
It follows that, for any sequence of $F_n$ (which satisfies the condition 
that that $k, m\in F_n$ and $k\not=m$ implies $|m-k|\ge 2$), the
sequence 
$\{a_{F_n}\}_{n\in \N}$ is full in $l^\infty(M(A)).$

One may construct a sequence $\{F^{(i)}\}$ of  infinite subsets 
of $\N$ such that, for each $i,$ if $m, n\in F^{(i)}$ are distinct, then  $|m-n|\ge 2,$ and
if $m\in F^{(i)},$ $n\in F^{(j)},$ and  $i\not=j,$   then $|n-m|\ge 2.$ 
Therefore, if $i\not=j,$  
\beq\label{Lold-11}
a_{F^{(i)}} a_{F^{(j)}}=0.
\eneq
For each $i\in \N,$ define $F^{(i)}_n=F^{(i)}\cap \{m\in \N: m\ge n\}.$ 
Put $a_n^{(i)}=a_{F^{(i)}_n},$ $n\in \N$ and $i\in \N.$ 
Then, 

(d) $\{a_n^{(i)}\}_{n\in \N}$ is full in $l^\infty(M(B))$  and, by \eqref{Lold-11},

(c)  $a_n^{(i)}a_n^{(j)}=a_n^{(j)}a_n^{(i)}=0\rforal n,\,\,{\rm when}\,\,i\not=j.$

For any $k\in \N$ and $d\in {\cal F}_k,$ 
 by \eqref{Lold-3}, 
\beq
\|a_n^{(i)}d-da_n^{(i)}\|<\sum_{i\in F_n^{(i)}} {2\over{2^i}}<1/2^{n-1}\to 0 \,\,\, {\rm as} \,\,n\to 0
\eneq
for all $i\in \N.$ 
It follows that
\beq
\lim_{n\to\infty}\|a_n^{(i)}d-da^{(i)}_n\|=0\rforal d\in \cup_{k=1}^\infty {\cal F}_k.
\eneq
Since $\cup_{k=1}^\infty {\cal F}_k$ is dense in the unit ball of $D,$  the above implies 
that
\beq\label{Lold-20}
\lim_{n\to\infty}\|a_n^{(i)}d-da^{(i)}_n\|=0\rforal d\in D\andeqn i\in \N.
\eneq
Thus 
$M(A)$ has property P3.

Moreover, by considering the sequence $\{\pi(a_n^{(i)})\},$ we see that 
$M(B)/B$ also has properties P1, P2 and P3.
\end{proof}

\begin{prop}\label{Pconscal}
Let $A$ be a non-unital but $\sigma$-unital \CA\, such that $M(A)/A$ is a simple \CA\, 
with property P1. 
Then $M(A)/A$ has  property  P3. 
\end{prop}

\begin{proof} We will show that $M(A)/A$ has property P3.
The proof is a modification of that of \ref{Lold}.
We will repeat some of the arguments.
Let $\pi: M(B)\to M(B)/B$ be the quotient map and $D$ a separable \SCA\, of $M(B).$
Let $\{d_k\}$ be a dense sequence in the unit ball of $D.$ 
Set ${\cal F}_n=\{d_1, d_2,...,d_n\}.$ 
Let $\{E_n\}$ be an approximate identity for $A$ such that 
\beq
E_{n+1}E_n=E_n\rforal n\in \N.
\eneq
Applying Proposition \ref{Pappid},  we choose a quasi-central approximate identity $\{e_n\}$
such that
\beq\label{Lconscal-2}
e_{n+1}e_n=e_n\andeqn \lim_{n\to\infty}\|e_nd-de_n\|=0\rforal d\in D.
\eneq
Moreover, there is a  subsequence $\{k(n)\}$ of $\N$ such that
\beq\label{Lconscale-3}
e_{n+1}\ge E_{k(n+1)},\,\,\, E_{k(n+1)}e_n=e_n\ge E_{k(n)},\,\,n\in \N.
\eneq
Note that $k(n+1)>k(n).$
Define $e_1'=e_1,$ $e_n'=e_{2n+1},$ $n\in \N.$
Then 
\beq\label{Lconscale-4}
&&e_{n+1}'e_n'=e_n'\andeqn\\\label{Lconscale-5+}
&&e_{n+1}'-e_n'=e_{2n+3}-e_{2n+1}\ge E_{k(2n+3)}-e_{2n+1}\ge E_{k(2n+3)}-E_{k(2n+2)}.
\eneq
Note  that, if $|n-m||\ge 2,$
\beq
(e_{n+1}'-e_n')(e_{m+1}'-e_m')=0.
\eneq
Choose an infinite subset $F\subset \N$ such that
if $n\not=m$ are in $F,$ then $|n-m|\ge 2.$ 
Define 
\beq
a_{F}=\sum_{n\in F}(e_{n+1}'-e_n')
\eneq
(which converges in the strict topology).
By \eqref{Lconscale-5+}, 
\beq
a_{F}\ge \sum_{n\in F} (E_{k(2n+3)}-E_{k(2n+2)})
\eneq
and the right term converges in the strict topology. 
Note that $\pi( \sum_{n\in F} (E_{k(2n+3)}-E_{k(2n+2)}))\not=0$ in $M(A)/A.$ 
Since $M(A)/A$ is simple and has property P1, the constant sequence 
$\{\pi(a_{F})\}$ is full in $l^\infty(M(A)/A).$ 

We  will choose $a_n^{(i)}=\pi(a_{F^{(i)}}),$  $n\in \N$ and 
$i\in \N.$  The rest of the proof carries with minimal notation modification.

\end{proof}

\begin{rem}\label{RemP1-3}

(1) Let $A$ be a non-unital and $\sigma$-unital simple \CA\, with continuous scale.
Then $M(A)/A$ is purely infinite and simple (see \cite[ Theorem 2.4 and Theorem 3.2]{Lincs2}). So it has properties P1 and P2.
By Proposition \ref {Pconscal}, $M(A)/A$ also has property P3.

(2) It follows from \ref{Lold}, \cite[Theorem 3.5]{Linfull} and part (1) of \cite[Proposition 3.11]{Linfull} 
that $B(l^2)$ has properties P1, P2 and P3. 

(3) If $A$ is a non-unital and $\sigma$-unital purely infinite simple \CA, 
then $M(A)$ and $M(A)/A$ have properties P1, P2 and P3 (see \cite[Proposition 3.4]{Linfull}).

There are other examples of \CA s satisfying properties P1, P2, and P3 (see \cite{Linfull}).
\end{rem}
%
%
%
%
\begin{prop}\label{Pcalkin}
Let $B_n$ be a unital purely infinite simple \CA\, which satisfies properties P1, P2 and P3,
$n\in \N.$
Then $q_\varpi(\{B_n\})$ has properties P1, P2 and P3. 
In particular, this holds for $B_n=B(l^2)/{\cal K}$ (for all $n\in \N$). 
\end{prop}

\begin{proof}
It follows from \cite[Proposition 2.5]{Linqsex} (see also \cite[Proposition 6.26]{MRbook} that
$q_\varpi(\{B_n\})$ is purely infinite and simple. Hence, $q_\varpi(\{B_n\})$ has properties 
P1 and P2.  It remains to show that $q_\varpi(\{ B_n\}) $ has property P3.

Put $C=q_\varpi(\{B_n\}).$  Fix  a separable \SCA\, $D\subset C.$

Let $D_n\subset B_n$ be a separable \SCA\,  such that
$\pi_{\varpi}(\{D_n\})\supset D.$

Fix $n\in \N.$  Since each $B_n$ has property P3, find  sequences $\{a^{(i)}_{n, k}\}_{k\in \N}$ such that

(1) $0\le a_{n,k}^{(i)}\le 1$ and $\|a_{n,k}^{(i)}\|=1,$ for all $k$ and $i.$

(2) $\lim_{k\to\infty}\|a_{n,k}^{(i)} d'_n-d'_n a_{n,k}^{(i)}\|=0\rforal
d'_n\in D_n\andeqn  i\in \N.$

(3) $\lim_{k\to\infty} \|a_{n,k}^{(i)}a_{n,k}^{(j)}\|=0,$ \,\, if $i\not=j.$

(4) $\{a_{n,k}^{(i)}\}_{k\in\N}$ is a full element in $l^\infty(B_n)$ for all $i\in \N.$

Let $\{d_k\}$ be a dense sequence in the unit ball of $D.$
Let $\{d_{k,n}\}_{n\in \N}\in l^\infty(\{B_n\})$ be such that $\|d_{k,n}\|=1$
for all $k, n\in \N$ and 
$\pi_\varpi(\{d_{k,n}\}_{n\in \N})=d_k,$ $k\in \N.$

For each $n\in \N,$ choose $k(n)\in \N$ such that
\beq\label{Pcalkin-8}
\|a_{n,k(n)}^{(i)} d_{j,n}-d_{j,n} a_{n, k(n)}^{(i)}\|<1/2^n
\eneq
for $1\le j\le n,$ $i\in \N.$  Moreover, 
\beq\label{Pcalkin-9}
\|a_{n,k(n)}^{(i)}a_{n,k(n)}^{(j)}\|<1/2^n,\,\,\, i\not=j
\eneq
for all $n\in \N.$

Let $b_n^{(i)}=(a_{n, k(n)}^{(i)}),$ $n\in \N$ and $i\in \N.$ 
Note that $\|b_n^{(i)}\|=1.$ Since $B_n$ is purely infinite simple, 
it has real rank zero (see \cite{Zh}). There is a non-zero projection 
$q_n^{(i)}\in B_n$ such that
\beq
b_n^{(i)} q_n^{(i)}\ge (1-1/2n)q_n^{(i)},\,\, n\in \N.
\eneq
Since $B_n$ is purely infinite simple, there is $w_n^{(i)}\in B_n$ such that
\beq\label{Pcalkin-10}
(w_n^{(i)})^*b_n^{(i)}w_n^{(i)}=1_{B_n}\andeqn \|w_n^{(i)}\|\le {1\over{1-1/2n}}
\eneq
$n\in \N$ and $i\in \N.$ Define $\{w_n^{(i)}\}\in l^\infty(\{B_n\}).$
Put 
\beq
c^{(i)}=\pi_\varpi(\{b_n^{(i)}\}).
\eneq
Define $c_n^{(i)}=c^{(i)},$ $n\in \N$ and $i\in \N.$
By \eqref{Pcalkin-10},  we have 

(d): $\{c_n^{(i)}\}$ is full in $l^\infty(q_\varpi(\{B_n\})).$

Also, we have 

(a): $0\le c_n^{(i)}\le 1.$

For any $k,$     by \eqref{Pcalkin-8},
\beq
c^{(i)}d_j=d_jc^{(i)}\rforal 1\le j\le k\andeqn i\in \N.
\eneq
It follows that

(b): $c_n^{(i)} d=dc_n^{(i)} \rforal d\in D$ and for all $n\in \N.$
Moreover, by \eqref{Pcalkin-9},   we have

(c): $c_n^{(i)}c_n^{(j)}=c_n^{(j)}c_n^{(i)}=0$ if $i\not=j$ for all $n\in \N.$

Hence (by (a), (b), (c) and (d) above) $q_\varpi(\{B_n\})$ has property P3.
\end{proof}

\section{Absorbing}

\begin{lem}\label{Lqoetient}
Let
$A$ be a separable 
amenable \CA\, with a fixed strictly positive element
$e_A\in A$ with $\|e_A\|=1.$ 
For any $\ep>0$ and any finite subset ${\cal F}\subset A$
with $e_A\in {\cal F},$  
there are $\dt>0$ and a finite subset ${\cal G}\subset A$ satisfying the following:
for any  \cpc\, $L: A\to B,$
where $B$ is a  unital purely infinite simple \CA\, which satisfies property P3,
such that  $\|L(e_A)\|=1,$ 
\beq
\|L(ab)-L(a)L(b)\|<\dt\tforal a, b\in {\cal G},
\eneq
then there is a non-zero \hm\, $h: A\to O_2\to B$ (factors through $O_2$)  and a partial isometry $u\in M_2(B)$
such that $u^*u=1_B,$ $uu^*=1_B\oplus q,$ where $q\in B$ is a projection 
such that $qh(a)=h(a)q=h(a)$ for all $a\in A,$ 
\beq
\label{C22}
&&\|L(a)\|\le  \|h(a)\|+\ep\|a\| \tforal a\in {\cal F}\tand\\\label{C22-1}
&&\|u^*\diag(L(a), h(a))u-L(a)\|<\ep\tforal a\in {\cal F}.
\eneq
Moreover,  if $A$ is unital, we may choose $u$ such that
$uu^*=1_B\oplus h(1_A).$
Furthermore, if $[1_B]=0$ in $K_0(B),$
we may further assume that $uu^*=1_B\oplus 1_B.$ 
\end{lem}

\begin{proof}
For the convenience, in the case that $A$ is unital,
we assume that $e_A=1_A.$

Assume that the lemma is false. 
Then there are $1/2>\ep_0>0$ and finite subset ${\cal F}_0\subset A$ satisfying     the following:
There exists a sequence of unital purely infinite simple \CA s $\{B_n\},$  a sequence of 
\cpc s $L_n: A\to B_n,$ a decreasing sequence   of positive numbers 
$r_n\searrow 0$ with $\sum_{n=1}^\infty r_n<\infty,$  an increasing sequence of finite subsets 
${\cal G}_n$ of the unit ball with $\cup_{n=1}^\infty {\cal G}_n$  dense in  the unit ball of $A,$ such that
$\|L_n(e_A)\|=1$ and 
\beq\label{Lquotient-1}
\|L_n(ab)-L_n(a)L_n(b)\|<r_n \rforal a, b\in {\cal G}_n,
\eneq
but
\beq\label{Lquotient-2}
\inf\{\max\{\|u_n^*\diag(L_n(a), h_n(a))u_n-L_n(a)\|: a\in {\cal F}_0\}\}\ge \ep_0,
\eneq
where the infimum is taken among all possible \hm s $h_n: A\to O_2\to B_n$
such that $\|L_n(a)\|\le \|h_n(a)\|+\ep_0\|a\|$  for all $a\in {\cal F}_0,$ 
all possible partial isometries  $u_n\in M_2(B_n)$  with
$u_n^*u_n=1_{B_n}$ and  $u_nu_n^*=1_{B_n}\oplus q_n,$ where $q_n$ is a projection in $B_n$ such that
\beq
h_n(a)q_n=q_nh(a)=h(a)\rforal a\in A,
%
\eneq
and $n\in \N.$
Moreover,  we may assume that $0\not\in {\cal F}_0,$ $e_A\in {\cal F}_0$ and ${\cal F}_0$ is in the unit ball of $A.$

Note that, by \eqref{Lquotient-1}, 
\beq\label{Lquotient-3}
\lim_{n\to\infty}\|L_n(ab)-L_n(a)L_n(b)\|=0\rforal a, b\in A.
\eneq

Denote by $\tilde A$ the minimal unitization  of $A.$ 
Define $\Lambda: \tilde A\to l^\infty(B)$ by $\Lambda(a)=\{L_n(a)\}_{n\in \N}$
and $\Lambda(1_{\tilde A})=\{1_{B_n}\}.$ 
Fix a free ultrafilter $\varpi\in \bt(\N)\setminus \N.$
Then, by \eqref{Lquotient-3}, $\pi_\varpi\circ \Lambda$  is a \hm\, from $A$ into $q_\varpi(B)=l^\infty(B)/c_{0,\varpi}(B).$
  Put $\psi':=\pi_\varpi\circ \Lambda$
and $I={\rm ker}\psi'.$ Let $C=\tilde A/I$ and $\pi_I: \tilde A\to C$ be the quotient map.  
Let $\psi: C\to q_\varpi(B)$  be the induced injective \hm\, by $\psi'$
(such that $\psi'=\psi\circ \pi_I$). We also have $\|\psi'(e_A)\|=1.$
In particular, $I\not=A.$  Moreover, in the case that $A$ is not unital, 
$C\not=1_{\tilde A}.$

Put $\eta=\min\{\ep_0/64, 1/64\}.$ 

Since $B$ is a unital purely infinite simple \CA, it follows from \cite[6.2.6]{MRbook} that
$q_\varpi(B)$ is a unital purely infinite simple \CA.  Therefore it has properties P1 and P2.
By the assumption, $B$ also has property P3.   By Proposition \ref{Pcalkin}, 
$q_\varpi(B)$  has property P3. 
Therefore, by \cite[Theorem 7.5]{Linfull},
there are injective \hm s $h_0: C\to O_2$ and $j_O: O_2\to q_\varpi(\{B_n\})$ 
and a partial isometry $v\in M_2(q_\varpi(\{B_n\}))$ such that
\beq\label{Lquotient-6}
&&v^*v=1_{q_\varpi(\{B_n\})},\,\, \, vv^*=1_{q_\varpi(\{B_n\})}\oplus j_O\circ h_0(1_C)\andeqn\\\label{Lquotient-7}
&&\|v^*\diag(\psi\circ \pi_I(a), j_0\circ h_0(\pi_I(a)))v-\psi\circ \pi_I(a)\|<\eta\rforal a\in {\cal F}_0
\cup\{1_{\tilde A}\}.
\eneq
Since both $\psi$ and $h_0$ are injective,
\beq\label{Lquotient-7+}
\|\psi\circ \pi_I(a)\|=\|h_0\circ \pi_I(a)\|\rforal a\in A.
\eneq
It follows that (see  \cite[Lemma 7.3]{Linsemiproj}, for example) that 
there is a non-zero \hm\, $J: O_2\to l^\infty(\{B_n\})$ 
(since $O_2$ is simple, $J$ is injective) such that $\pi_\varpi\circ J=j_O.$ 
Define $H: A\to l^\infty(\{B_n\})$ by $H(a)=J\circ h_0\circ \pi_I(a)$ for all $a\in A.$ 
We may write $J=\{J_n\},$ where each $J_n: O_2\to B_n$ is a non-zero \hm.
Put $q_n=J_n\circ h_0(1_C),$ $n\in \N.$ Since both $J_n$ and $h_0$ are \hm s, 
$q_n$ is a projection, $n\in \N.$ 
There is also a sequence of elements $v_n\in M_2(B_n)$
such that $\pi_\varpi(\{v_n\})=v.$

By \eqref{Lquotient-6},\eqref{Lquotient-7} and \eqref{Lquotient-7+},  there is ${\cal X}\in \varpi$
such that
\beq\label{Lquotient-8}
&&\|v_n^*v_n-1_{B_n}\|<\eta,\,\,
\|v_nv_n^*-(1_{B_n}+q_n)\|<\eta,\\\label{Lquotient-9}
&&\|L_n(a)\|\le \| h_0\circ \pi_I(a)\|+({\ep_0\over{2}})\|a\|\rforal a\in {\cal F}_0\andeqn\\\label{Lquotient-10}
&&\hspace{-0.5in}\|v_n^*\diag(L_n(a), J_n\circ h_0\circ \pi_I(a))v_n-L_n(a)\|<\eta+\min\{\ep_0/64, 1/64\}
\rforal a\in {\cal F}_0
\eneq
and for all $n\in {\cal X}.$   We note that 
\beq\label{Lqt-8+}
\|h_0\circ \pi_I(e_A)\|\ge \|L_n(e_A)\|-\ep_0\|e_A\|\ge 1/2.
\eneq
Put $p_n=1_{B_n}\oplus q_n,$ $n\in \N.$
By replacing $v_n$ by $p_nv_n1_{B_n},$ we may assume $p_nv_n1_{B_n}=v_n$
for all $n\in {\cal X}.$  Then there is, for each $n\in {\cal X},$ 
a partial isometry $u_n\in M_2(B_n)$ such that
\beq\label{Lquotient-12}
\|v_n-u_n\|<2\eta,\,\,\, u_n^*u_n=1_{B_n}\andeqn u_nu_n^*=p_n
\eneq
for all $n\in {\cal  X}.$  Put $h_n=J_n\circ h_0\circ \pi_I: A\to O_2\to B_n$
for $n\in {\cal X}.$ 
Then 
\beq\label{Lquotient-14}
\|L_n(a)\|\le \|h_n(a)\|+\ep_0\|a\|\rforal a\in {\cal F}_0.
\eneq
It follows from \eqref{Lquotient-10} that (recall that ${\cal F}_0$ is in the unit ball of $A$) 
\beq\label{Lquotient-15}
\|u_n^*\diag(L_n(a), h_n(a))u_n-L_n(a)\|<4\eta+4\min\{\ep_0/64, 1/64\}<\ep_0
\eneq
for all $n\in {\cal X}.$  Then, \eqref{Lquotient-15} and \eqref{Lquotient-14} 
contradict with \eqref{Lquotient-2}.   We also note that, 
in  case that $A$ is unital;, $u_nu_n^*=1_{B}\oplus h_n(1_A),$ by \eqref{Lqt-8+}, ${h_n}|_{A}\not=0$ 
in the case that $A$ is non-unital,  and 
\beq
&&u_nu_n^*\diag(L_n(a), h_n(a))=p_n\diag(L_n(a), h_n(a))\\
&&=\diag(L_n(a), h_n(a))p_n=\diag(L_n(a), h_n(a))\rforal a\in A.
\eneq 
Thus the lemma follows once we deal with the  ``Furthermore" part.

To see the ``Furthermore" part, let us assume that $[1_B]=0.$
It follows from the existence of $u$ above that $[q]=0.$ 
Since $B$ is purely infinite,  one obtains a partial isometry $w\in B$ such 
that
\beq
w^*w=q\andeqn ww^*=1_B.
\eneq
Define $h_2: A\to O_2\to B$ by $h_2(a)=wh(a)w^*$ for all $a\in A$ 
and $U:=(1_B\oplus w)u.$ Then 
\beq
U^*U=u^*(1_B\oplus w^*w)u=u^*(1_B\oplus q)u=1_B\andeqn\\
UU^*=(1_B\oplus w)uu^*(1_B\oplus w^*)=(1_B\oplus 1_B).
\eneq
Moreover, by \eqref{C22-1},
\beq
U^*\diag(L(a), h_2(a))U&=& u^*\diag(L(a), w^*wh(a)ww^*)u\approx_{\ep} L(a)
\eneq
for all $a\in {\cal F}.$  

Then we choose $U$ instead of $u$ and $h_2$ instead of $h.$ Note that 
\eqref{C22} also holds by replacing $h$ by $h_2.$
\end{proof}

\begin{cor}\label{MC1}
Let
$A$ be a separable 
amenable \CA. 
For any $\ep>0,$  any finite subset ${\cal F}\subset A,$ any $\sigma>0,$ 
there are $\dt>0,$ finite subsets ${\cal G}, {\cal H}\subset A$ 
satisfying the following:
for any non-zero \cpc\, $L: A\to B,$
where $B$ is a  unital purely infinite simple \CA\, which satisfies property P3,
such that 
\beq\label{MC1-1}
\|L(ab)-L(a)L(b)\|<\dt\tforal a, b\in {\cal G}\tand 
\|L(c)\|\ge \sigma\|c\| \tforal c\in {\cal H},
\eneq
then, there is an injective \hm\, $h: A\to O_2\to B$ (factors through $O_2$)  and a partial isometry $u\in M_2(B)$
such that $u^*u=1_B,$ $uu^*=1_B\oplus q$ for some projection $q\in B$ 
with $qh(a)=h(a)q=h(a)$ for all $a\in {\cal F},$ and 
\beq
&&\|u^*\diag(L(a), h(a))u-L(a)\|<\ep\tforal a\in {\cal F},
\eneq
If $A$ is unital, we may choose $q=h(1_A).$  Moreover, if $[1_B]=0$ in $K_0(B),$
we may choose $q=1_B.$
\end{cor}

\begin{proof}
Almost the same proof as that of Lemma \ref{Lqoetient} applies.
In the proof of Lemma \ref{Lqoetient}, we will ignore anything related to condition \eqref{C22}.
However,  the second part of condition \eqref{MC1-1} implies that, in addition to \eqref{Lquotient-3},
we may also  assume that 
\beq
\|L_n(a)\|\ge \sigma \|a\|\rforal a\in {\cal G}_n.
\eneq
Since we assume that ${\cal G}_n\subset {\cal G}_{n+1}$ for all $n\in \N.$
This implies that, for each $m\in \N,$
\beq
\|\pi_\varpi\circ \Lambda(a)\|\ge \sigma \|a\|\rforal a\in {\cal G}_m.
\eneq
Hence 
\beq
\|\pi\circ \Lambda(a)\|\ge \sigma \|a\|\rforal a\in \cup_{n=1}^\infty{\cal G}_n.
\eneq
It follows that 
\beq
\|\pi\circ \Lambda (a)\|\ge \sigma \|a\|\rforal a\in A.
\eneq
In other words, $\pi\circ L$ is injective and  $I={\rm ker}\psi'=\{0\}.$
Hence $C=A$ in the proof of Lemma \ref{Lqoetient}. So $h_0$ obtained in the proof of Lemma \ref{Lqoetient} is injective.
Thus the corollary follows the rest of the proof of Lemma \ref{Lqoetient}.
\end{proof}

\begin{rem}\label{RMC1}
One may notice, from the proof, that ${\cal H}$ depends on $\ep$ as well as ${\cal F},$
in particular, when $\ep$ is small and ${\cal F}$ is large,  ${\cal H}$ is also large.  But it does not depend on $\sigma.$  This feature does not, however,  seem to be very helpful. 
\end{rem}

\begin{cor}\label{MC2}
Let $A$ be a separable amenable \CA, $\ep>0$ and 
${\cal F}\subset A$ be a finite subset.
Then there are $\dt>0,$ a finite subset ${\cal G}\subset A$ satisfying the following:

Suppose that $L: A\to B,$ where $B$ is any 
unital purely infinite simple \CA\, with properties P3 such that $[1_B]=0$
in $K_0(B),$ 
is a \cpc\, such that
\beq
\|L(ab)-L(a)L(b)\|<\dt\tforal a,b\in {\cal G}.
\eneq
Then, for any integer $k\in \N,$  there exists a partial isometry $u\in M_{k+1}(B)$
and a \hm\, $h: A\to O_2\to B$ such that
\beq
&&\|L(a)\|\le \|h(a)\|+\ep\|a\| \rforal a\in {\cal F},\\
&&u^*u=1_B,\,\,\, uu^*=1_{M_{k+1}(B)}\tand\\
&&\|u^*\diag(L(a), \overbrace{h(a),h(a),...,h(a)}^k)u-L(a)\|<\ep\tforal a\in {\cal F}.
\eneq
\end{cor}

\begin{proof}
We may assume that $L\not=0.$

Applying Lemma \ref{Lqoetient}, we obtain $\dt>0,$
finite subset ${\cal G}\subset A,$ an injective  \hm\, $j_O: O_2\to B$
and a \hm\, $h_0: A\to O_2,$ a partial isometry $w\in M_2(B)$ and a projection 
$q\in B$  such that
\beq
&&\|L(a)\|\le \|h_1(a)\|+({\ep\over{4}})\|a\|\rforal a\in {\cal F},\\\label{MC2-9}
&&\|w^*\diag(L(a), j_O\circ  h_0(a))w-L(a)\|<\ep/4\rforal a\in {\cal F},\\\label{MC2-10}
&&w^*w=1_B,\,\, ww^*=1_B\oplus q\andeqn qh_1(a)=h_1(a)q=h_1(a)\rforal a\in A,
\eneq
where $h_1=j_O\circ h_0.$
If $A$ is unital, we may choose $q=h_1(1_A).$  In that case,
let $h_0(1_A)=e\in O_2$ 
Note that we have assumed that $q=h(1_A).$ 
We observe that $O_2\cong eO_2e.$
Since $K_0([1_B])=0$ and $B$ is purely infinite simple, there is a unital embedding $j'_O: eO_2e\to B.$ 
Consider $j_O', {j_O}|_{eO_2}: eO_2 e\to B.$ Since $O_2$ is $KK$-contractive, 
there is a partial isometry $v\in B$ with $v^*v=1_B$ and $vv^*=j_O(e)=q$
such that
\beq\label{MC2-12}
\|v(j'_O\circ h_0(a))v^*-j_O\circ h_0(a)\|<\ep/4\rforal a\in {\cal F}
\eneq
(see \cite[Theorem 6.7]{LinTAMS04}, for example).
We have
\beq
&&\hspace{-0.4in}(1_B\oplus v^*)w w^*(1_B\oplus v)=
(1_B\oplus v^*)(1_B\oplus q)(1_B\oplus v)=1_B\oplus v^*qv=1_B\oplus 1_B\andeqn\\
&&\hspace{-0.4in}w^*(1_B\oplus v)(1\oplus v^*)w=w^*(1_B\oplus q)w=1_B.
\eneq
Put $w_1=(1_B\oplus v^*)w\in M_2(B).$ Then, by \eqref{MC2-12} and \eqref{MC2-10},
for all 
$a\in {\cal F},$
\beq
w_1^*\diag(L(a), j_O'\circ h_0(a))w_1\approx_{\ep/4} w\diag(L(a), j_O\circ h_0(a))w\approx_{\ep/4} 
L(a).
\eneq

In the case that $A$ is not unital,  we may assume that $q=j_O(1_{O_2})\in B.$
Choose a unital embedding $j_O': O_2\to B.$
Then, there is a partial isometry $v'\in B$ such that
\beq
&&{v'}^*v'=1_B,\,\, v'{v'}^*=q\andeqn\\
&&\|{v'}j_O'(h_0(a)){v'}^*-j_O(h_0(a))\|<\ep/4\rforal a\in {\cal F}.
\eneq
Define $w_2=(1_B\oplus  {v'}^*)w.$ 
Then
\beq
&&w_2w_2^*=(1_B\oplus {v'}^*)ww^*(1_B\oplus v')=1_B\oplus {v'}^*qv'=1_B\oplus 1_B\andeqn\\
&&w_2^*w_2=w^*(1_B\oplus v')(1_B\oplus {v'}^*)w=w^*(1_B\oplus q)w=1_B.
\eneq
Moreover, for all $a\in {\cal F},$
\beq
w_2^*\diag(L(a), j_O'(h_0(a)))w_2\approx_{\ep/4} w^*\diag(L(a), j_O(h_0(a)))w\approx_{\ep/4} 
L(a).
\eneq
Therefore, to simplify the notation, \wilog, for the rest of the proof, we may assume 
in \eqref{MC2-10}  that $q=1_B.$

Consider the map $\phi: O_2\to M_k(O_2)$ by
\beq
\phi(x)=\diag(\overbrace{x,x,...,x}^k)\rforal x\in O_2.
\eneq
Since $O_2$ is $KK$-contractive, 
there is  a partial isometry $v_2\in M_k(O_2)$
such that $v_2^*v_2=1_{O_2},$ $v_2v_2^*=1_{M_k(O_2)}$ and 
\beq\label{MC2-15}
\|v_2^*\phi(x)v_2-x\|<\ep/2\rforal x\in h_0({\cal F}).
\eneq
Define  $J_O: O_2\to M_k(B)$  by $J_O=j_O\circ \phi$ and $u=(1_B\oplus J_O(v_2))w.$
Then (recall that we now assume that $ww^*=1_{M_2(B)}$)
\beq
u^*u&=&w^*(1_B\oplus J_O(v_2)^*)(1_B\oplus J_O(v_2))w=w^*(1_B\oplus 1_B)w=1_B\andeqn\\
uu^*&=&(1_B\oplus J_O(v_2))ww^*(1_B\oplus J_O(v_2^*))=1_B\oplus J_O(v_2)^*J_O(v_2)=1_{M_{k+1}(B)}.
\eneq
Moreover, by \eqref{MC2-15} and \eqref{MC2-9}, 
\beq\label{MC2-16}
u^*\diag(L(a), J_O\circ h_0(a))u\approx_{\ep/2} w^*\diag(L(a), j_O(h_0(a))w\approx_{\ep/4} L(a).
\eneq
for all $a\in {\cal F}.$
Define $h=j_O\circ h_0.$  Then, by \eqref{MC2-16}, we have
\beq
\|u^*\diag(L(a),\overbrace{h(a),h(a),...,h(a)}^k)u-L(a)\|<\ep\rforal a\in {\cal F}.
\eneq
\end{proof}

\begin{cor}\label{MC3}
Let $A$ be a separable amenable \CA, $\ep>0,$ 
${\cal F}\subset A$ be a finite subset and $\sigma>0.$
Then there are $\dt>0,$  a finite subset ${\cal G},\,\,{\cal H}\subset A$
  satisfying the following:

Suppose that $L: A\to B,$ where $B$ is any 
unital purely infinite simple \CA\, with properties P3 such that $[1_B]=0$
in $K_0(B),$ 
is a \cpc\, such that
\beq
\|L(ab)-L(a)L(b)\|<\dt\tforal a, b\in {\cal G}\tand \|L(c)\|\ge \sigma\|c\| \tforal c\in {\cal H}.
\eneq
Then, for any integer $k\in \N,$  there exists a partial isometry $u\in M_{k+1}(B)$
and an injective  \hm\, $h: A\to O_2\to B$ such that
\beq
&&u^*u=1_B,\,\,\, uu^*=1_{M_{k+1}(B)}\tand\\
&&\|u^*\diag(L(a), \overbrace{h(a),h(a),...,h(a)}^k)u-L(a)\|<\ep\tforal a\in {\cal F}.
\eneq
\end{cor}

\begin{proof}
The proof is a modification of that of \ref{MC2} following the same way as that of 
\ref{MC1}.
\end{proof}



\section{Quasidiagonal \CA s}

\begin{df}\label{Dquasi}
Let $A$ be a separable \CA\, and $H$ be a separable infinite dimensional Hilbert space.
 Recall that a representation 
$\phi: A\to B(H)$ is said to be quasidiagonal, if there exists 
an approximate identity $\{p_n\}$ of ${\cal K}$ consisting of projections 
such that
\beq
\lim_{n\to\infty}\|p_n\phi(a)-\phi(a)p_n\|=0\tforal a\in A.
\eneq

A \CA\, is said to be quasidiagonal if it admits a faithful quasidiagonal extension.
\end{df}

Let us state a Voiculescu theorem as follows.

\begin{lem}\label{Lquasi-1}
Let $A$ be a separable quasidiagonal \CA. 
Suppose that $\phi: A\to B(H)$ is a faithful representation  such that
$\phi(A)\cap {\cal K}=\{0\}.$ Then $\phi$ is quasidiagonal. 
\end{lem}

\begin{proof}
Let $h: A\to B(H)$ be a faithful quasidiagonal representation.
Consider the countable direct sum $H_\infty=\bigoplus H$, which is also an infinite 
dimensional separable Hilbert space. Therefore there 
is a unitary $u$ from $H_\infty$ onto $H.$ 
Define $h_\infty: A\to B(H_\infty)$ by 
\beq
h_\infty(a)=h(a)\oplus h(a)\oplus \cdots \oplus h(a)\oplus \cdots \,\,\,\,\, \rforal a\in A.
\eneq
Consider $h_1: A\to B(H)$ by 
$h_1(a)=u h_\infty(a)u^*$ for all $a\in A.$
Since $h$ is quasidiagonal representation,  there is a sequence of increasing 
finite rank projections $\{p_n\}$ which is an approximate identity 
for ${\cal K}$ such that
\beq
\lim_{n\to\infty}\|p_nh(a)-h(a)p_n\|=0\rforal a\in A.
\eneq
Now let $\{a_k\}$ be a dense subset of $A.$ 
Suppose that
\beq
\|p_nh(a_k)-h(a_k)p_n\|<1/2^n\rforal k=1,2,...,n.
\eneq
Now let 
$q_1=p_1,$ 
\beq
q_n=u\diag(\overbrace{p_n\oplus p_n\oplus \cdots p_n}^n)u^*,\,\,n\in \N.
\eneq
Then, for $n>1,$
\beq
\|q_nh_1(a_k)-h_1(a_k)q_n\|<1/2^n,\,\,\, k=2, 3,...,n.
\eneq
It follows that
\beq
\lim_{n\to\infty}\|q_n h_1(a)-h_1(a)q_n\|=0\rforal a\in A.
\eneq
Note that $q_n\to 1$ in strong operator topology.  
Hence $h_1$ is a faithful quasidiagonal representation.
Moreover,  $h_1(A)\cap {\cal K}=\{0\}.$
It follows Voiculescu's Weyl - von Neumann theorem （\cite{DV1})  that 
$\phi$ is also a quasidiagonal representation. 
\end{proof}

We would also like state the following corollary:

\begin{lem}\label{Lquasi-2}
Let $A$ be a separable  \CA. 
Suppose that $\phi: A\to B(H)$ is a faithful representation  such that
$\phi(A)\cap {\cal K}=\{0\}.$ Then, for any $k\in \N,$
there is a sequence of partial isometries $w_n\in M_{k+1}(B(H))$ 
such that 
\beq
&&w_n^*w_n=1_{B(H)},\,\,\, w_nw_n^*=1_{M_{k+1}(B(H))},\\
&&w_n^*\diag(\overbrace{\phi(a), \phi(a),...,\phi(a)}^k)w_n-\phi(a)\in {\cal K}\andeqn\\
&&\lim_{n\to\infty}\|w_n^*\diag(\overbrace{\phi(a), \phi(a),...,\phi(a)}^k)w_n-\phi(a)\|=0
\eneq
for all $a\in A.$ Moreover, if $A$ is quasidiagonal, then we may assume 
that $\phi$ is also quasidiagonal.
\end{lem}

{\bf The proof of Theorem \ref{Main}}:

Let $\ep>0$ and a finite subset ${\cal F}\subset A$ be given.
We may assume that ${\cal F}$ is in the unit ball of $A.$
Fix $\sigma>0.$    

Choose  ${\cal G}_1\subset A$ and $\dt_1>0$ as  ${\cal G}$ and $\dt$
in Lemma \ref{Luniqh}  for $\ep/16$ and ${\cal F}.$ 
\Wlog\, (with  possibly smaller $\dt_1$), we may assume that ${\cal F}\subset {\cal G}_1$ 
and ${\cal G}_1$ is in the unit ball of $A.$

Put $\ep_1=\min\{\dt_1/4, \ep/16\}>0.$ 

Choose a finite subset ${\cal G}_2\subset A$ (as ${\cal G}$) and a positive number $\dt_2$ (as $\dt$) 
and an integer $k$  
in Lemma \ref{Luniq} for ${\cal G}_1$ (as ${\cal F}$),  $\ep_1$ (as $\ep$) and $\sigma$ 
(as $\lambda$).
We may assume that ${\cal G}_1\subset {\cal G}_2$ and both are in the unit ball of $A.$

Put $\ep_2=\min\{\ep_1/4, \dt_2/16\}>0.$

Next choose a  finite subset ${\cal G}_3\subset A,$  a positive number  $\dt_3$  and an integer $k_1\in \N$ 
as ${\cal G},$ $\dt$  and  $k\in \N$ in 
Lemma \ref{Luniq}    for ${\cal G}_2$ (in place of ${\cal F}$) and $\ep_2/32$ (in place of $\ep$). 
We may assume that ${\cal G}_2\subset {\cal G}_3,$  and both are in the unit ball  of $A$
and $k_1\ge k+1.$ 

Put $\ep_3=\min\{\ep_2/2, \dt_3/4\}>0.$

Then choose finite subsets ${\cal G}_4,\,\, {\cal H}\subset A$ and  $\dt_4>0$  
as ${\cal G},$ ${\cal H}$ and  $\dt$  in
Corollary
\ref{MC1}    for ${\cal G}_3$ (in place of ${\cal F}$) and $\ep_3/32$ (in place of $\ep$)
(as well as for the given $\sigma$). 
We may assume, \wilog,  that ${\cal G}_3\subset {\cal G}_4$ and both are in the unit ball of $A.$

Let  $\dt=\min\{\dt_4/2, \ep_3/4\}>0$ and ${\cal G}={\cal G}_4.$

Suppose that $L$ is as in the lemma for ${\cal G}$ and $\dt.$
By applying Proposition \ref{PC-E},
we may further assume, \wilog, by choosing even smaller $\dt$ and larger 
${\cal G},$ if necessarily,  that $L$ is a \cpc.

Consider $\pi\circ L: A\to B(H)/{\cal K}.$ 
By Lemma \ref{Pcalkin}, $B(H)/{\cal K}$  is a unital purely infinite simple \CA\, satisfying 
property P3. 
Put $B=B(H)/{\cal K}.$  Note that $K_0(B)=\{0\}.$  Applying Corollary  \ref{MC1}, we obtain an  injective \hm\, 
$j_O: O_2\to B,$  an injective \hm\, $h_0: A\to O_2$ 
and a partial isometry $v_1\in M_{2}(B)$ 
such that 
\beq
&&v_1^*v_1=1_B,\,\,\, v_1v_1^*=1_{M_{2}(B)},\\\label{pMT-4}
&&v_1^*\diag(\pi\circ L(a), h_1(a))v_1\approx_{\ep_3/32}\pi\circ L(a)\rforal a\in {\cal G}_3,
%
\eneq
where $h_1=j_O\circ h_0.$ 

Let $Q_0\in B(H)$ be a projection such that $\pi(Q_0)=v_1^*\diag(1_B,0)v_1.$
Let $Q_1=1_{B(H)}-Q_0.$
Note that $\pi(Q_1)=v_1^*\diag(0, 1_B)v_1.$ 
Since $O_2$ is $KK$-contractive, there is a \hm\, 
$J_O: O_2\to Q_1B(H)Q_1$ such that  $\pi\circ J_O(c)=v_1^*j_O(c)v_1\rforal  c\in O_2.$
Put  $h_2=J_O\circ h_0: A\to Q_1B(H)Q_1.$ Note that $\pi\circ h_2={\rm Ad}v_1\circ j_O\circ h_0={\rm Ad}v_1\circ h_1$ is injective.
Consider ${\bar h}_2: A\to M_{k_1}(Q_1B(H)Q_1)$
by 
\beq
{\bar h}_2(a)=\diag(\overbrace{h_2(a),h_2(a),...,h_2(a)}^{k_1})\rforal a\in A.
\eneq
Let $w_1\in M_{k_1}(B(H))$ such that 
$w_1^*w_1=Q_1$ and $w_1w_1^*=1_{M_{k_1}(Q_1B(H)Q_1)}.$ 
Note $Q_1B(H)Q_1=B(Q_1H).$
By Voiculescu's  Weyl-von Neumann theorem (\cite{DV1}),  there is a unitary $w_2\in Q_1B(H)Q_1$
such that
\beq\label{MM-9}
w_2^* w_1^* {\bar h}_2(a)w_1w_2-h_2(a)\in Q_1{\cal K}Q_1\rforal a\in A.
\eneq
Put $u_1=( Q_0\oplus w_1w_2)\in M_{k_1+1}(B(H)).$
Then 
$u_1^*u_1=Q_0\oplus w_2^*w_1^*w_1w_2=Q_0\oplus Q_1=1_{B(H)}$
and $u_1u_1^*=Q_0\oplus 1_{M_{k_1}(Q_1B(H)Q_1)}.$
Put $$C=(Q_0\oplus 1_{M_{k_1}(Q_1B(H)Q_1)})M_{k_1}(B(H))(Q_0\oplus 1_{M_{k_1}(Q_1B(H)Q_1)}).$$


Recall that  $h_2(A)\cap {\cal K}=\{0\}.$
By the choice of $k_1,$ as well as $\dt_3$ and ${\cal G}_3,$  applying Lemma \ref{Luniq},  there  is a unitary $u_2\in M_{k_1+1}(B)$ 
such that
\beq\label{MM-15}
u_2^*\diag(h_2(a), \overbrace{h_2(a),h_2(a),...,h_2(a)}^{k_1})u_2\approx_{\ep_2/32} 
\diag(L(a), \overbrace{h_2(a),h_2(a),...,h_2(a)}^{k_1})
\eneq
for all $a\in {\cal G}_2.$ 
Define 
\beq
c_a:=L(a)-u_1^*u_2^*\diag(h_2(a), {\bar h}_2(a))u_2u_1\rforal a\in A.
\eneq
Then, by \eqref{pMT-4}, \eqref{MM-9} and \eqref{MM-15},
\beq
\|\pi(c_a)\|<\ep_3/32+\ep_2/32\rforal a\in {\cal G}_2.
\eneq


Applying  Lemma  \ref{Lquasi-1}, we obtain an approximate identity $\{e_n\}$ for ${\cal K}$ 
consisting of finite rank projections such that
\beq
\lim_{n\to\infty}\|e_n h_2(a)-h_2(a)e_n\|=0\rforal a\in A.
\eneq
Put $p_n=\diag(\overbrace{e_n, e_n,...,e_n}^{k_1+1})$ 
and $p_n'=\diag(0, \overbrace{e_n, e_n,...,e_n}^{k_1}),$ $n\in \N.$
Then $\{p_n\}$ is an approximate identity for $M_{k_1+1}({\cal K})$ and 
$\{q_n:=u_1^*u_2^*p_nu_2u_1\}$ is an approximate identity consisting of  finite rank projections
for ${\cal K}.$
Choose a sufficiently large $n,$
we may assume that
\beq\label{MM-20}
\|(1-q_n) c_a\|<\ep_3/16+\ep_2/32,\,\, \|c_a(1-q_n)\|<\ep_3/16+\ep_2/32\andeqn\\ 
\|(1-q_n)c_a(1-q_n)\|<\ep_3/16+\ep_2/32
\rforal a\in {\cal G}_2. 
\eneq
We choose  an even larger $n$ such that
\beq\label{MM-22}
\|e_nh_2(a)-h_2(a)e_n\|<\ep_2/32\rforal a\in {\cal G}_2.
\eneq
On ${\cal G}_2,$ by \eqref{MM-22}
\beq\nonumber
&&\hspace{-0.4in}(1-q_n) u_1^*u_2^*\diag(h_2(a), {\bar h}_2(a))u_2u_1=u_1^*u_2^*\diag((1-e_n)h_2(a), (1-p_n'){\bar h}_2(a))u_2u_1\\\label{MM-21}
&&\approx_{\ep_2/32}u_1^*u_2^*\diag(h_2(a), {\bar h}_2(a))u_2u_1(1-q_n).
\eneq
It follows from \eqref{MM-20} and \eqref{MM-21} that, for all $a\in {\cal G}_2,$ 
\beq\label{MM-23}
&&(1-q_n)L(a)\approx_{\ep_3/16+\ep_2/16} L(a)(1-q_n),\\\label{MM-23+1}
&&q_nL(a)\approx_{\ep_3/16+\ep_2/16} L(a)q_n\andeqn\\
&&\hspace{-0.4in}(1-q_n)L(a)(1-q_n)\\
&&\approx_{\ep_3/16+\ep_2/32} u_1^*u_2^*\diag((1-e_n)h_2(a)(1-e_n),
(1-p_n'){\bar h}_2(a)(1-p_n'))u_2u_1.
\eneq
Hence, for all $a\in {\cal G}_2,$ 
\beq\nonumber
&&\hspace{-0.7in}L(a)\approx_{\ep_3/8+\ep_2/8} (1-q_n)L(a)(1-q_n)+q_nL(a)q_n\\\label{MM-26}
&&\hspace{-0.5in} \approx_{\ep_3/16+\ep_2/32}  u_1^*u_2^*\diag((1-e_n)h_2(a)(1-e_n),
(1-p_n'){\bar h}_2(a)(1-p'_n))u_2u_1 +q_nL(a)q_n.
 \eneq
  Define, for all $a\in A,$ 
 \beq\label{MM-27}
 L_0(a)=\diag(1-e_n)h_2(a)(1-e_n),\overbrace{0,0,...,0}^{k_1})+u_1u_2q_nL(a)q_nu_2^*u_1^*.
 \eneq
 Since $\pi(e_n)=0,$
 there is a  \hm\, 
 \beq
 h_3: A\to ((1-e_n)\oplus p_n)M_{k+1}(B(H))((1-e_n)\oplus p_n).
 \eneq
 such that $\|\pi\circ h_3(a)\|=\|a\|$ for all $a\in A.$
 Note that
 \beq
2(\ep_3/16+\ep_2/16+\ep_2/32)<\dt_2.
\eneq
Hence,  by \eqref{MM-22} and \eqref{MM-23+1}, 
 $L_0$ is ${\cal G}_2$-$\dt_2$-multiplicative. 
Moreover, by \eqref{MM-22},  $(1-e_n)h_2(a)(1-e_n)$ is also ${\cal G}_2$-$\dt_2$-multiplicative. 

 Since 
 $\|\pi\circ h_2(a)\|=\|h_0(a)\|=\|a\|$ for all $a\in A,$  by applying Lemma \ref{Luniq}
 again, 
 there is a unitary $u_3\in M_{k+1}(B(H))$  such that
 \beq\nonumber
 &&\hspace{-0.4in}L_0(a)\oplus \diag(\overbrace{(1-e_n)h_2(a)(1-e_n),...,(1-e_n)h_2(a)(1-e_n)}^{k_1})\\\label{MM-27+}
 &&\approx_{\ep_1} u_3^*(h_3(a)\oplus \diag(\overbrace{(1-e_n)h_2(a)(1-e_n),...,(1-e_n)h_2(a)(1-e_n)}^{k_1}))u_3
 \eneq
 for all $a\in {\cal G}_1.$ In other words,
 \beq\label{MM-28}
 \diag(L_0(a), (1-p_n'){\bar h}_0(a)(1-p_n'))\approx_{\ep_1} u_3^*\diag(h_3(a), (1-p_n'){\bar h}_2(a)(1-p_n'))u_3.
 \eneq
 
 Put $P=((1-e_n)\oplus p_n).$ 
 Note that
 there is an injective \hm\, $H_1: A\to (1-P)M_{k+1}(B(H))(1-P).$
 Since $\|\pi\circ h_3(a)\|=\|a\|$ for all $a\in A,$
 $h_3$ is full \hm\, to $PM_{k+1}(B(H))P\cong B(H).$
  By  Lemma \ref{Luniqh}, 
 there is a unitary $u_4\in M_{k+1}(B(H))$ 
 such that
 \beq\label{MM-30}
 \diag(h_3(a), (1-p'_n){\bar h}_2(a)(1-p_n'))\approx_{\ep/4} u_4^*\diag(h_3(a), H_1(a))u_4\rforal a\in {\cal F}.
 \eneq
 Define $H_2: A\to B(H)$  by $H_2(a)=u_1^*u_2^*u_3^*u_4^*\diag(h_3(a), H_1(a))u_4u_3u_2u_1.$
 Then $H_2$ is a \hm. Moreover $\pi\circ H_2$ is injective.
 
 Put $\eta_1=\ep_3/8+\ep_2/8+\ep_3/16+\ep_2/32.$
We compute that, applying \eqref{MM-26}, \eqref{MM-27}, \eqref{MM-28} and \eqref{MM-30}, 
\beq
L(a) &\approx_{\eta_1}&u_1^*u_2^*\diag(L_0(a), (1-p_n'){\bar h}_2(a)(1-p_n'))u_2u_1\\
&\approx_{\ep_1}& u_1^*u_2^*u_3^*\diag(h_3(a), (1-p'){\bar h}_2(a)(1-p_n')u_3u_2u_1\\
&\approx_{\ep/4}&u_1^*u_2^*u_3^*u_4^*\diag(h_3(a), H_1(a))u_4u_3u_2u_1\\
&=&H_2(a)\hspace{0.4in}\rforal a\in {\cal F}.
\eneq
 
Since 
\beq
\eta_1+\ep_1+\ep/4&\le &(3\ep_3/16+5\ep_2/32)+\ep/16+\ep/4\\
  &\le & (3\ep_2/32+5\ep_1/128)+5\ep/16\\
  &\le & (3\ep_1/128+5\ep_1/128)+5\ep/16<\ep,
\eneq 
 the proof is complete (by letting $h=H_2$).

\begin{lem}\label{Lsimple}
Let $A$ be a separable amenable simple \CA. 
Then, for any finite subset ${\cal H}\subset A,$ 
and any $0<\sigma<1,$ 
there exist $\eta>0$ and a finite subset ${\cal G}\subset A$ satisfying the following:
If $L: A\to B$ (any \CA s $B$) is a contractive positive linear map  with $\|L\|\ge \sigma$ 
and 
\beq
\|L(ab)-L(a)L(b)\|<\eta\tforal a, b\in {\cal G},
\eneq
then 
\beq
\|L(c)\|\ge (1/2)\|c\|\tforal c\in {\cal H}.
\eneq
\end{lem}

\begin{proof}
Otherwise, 
one obtains a finite subset ${\cal H}_0\subset A\setminus \{0\}$ and $0<\sigma_0<1,$ 
a sequence of \CA s\, $\{B_n\}$ and a sequence of contractive positive linear maps $L_n: A\to B_n$
such that
\beq
\lim_{n\to\infty}\|L_n(ab)-L_n(a)L_n(b)\|=0\rforal a, b\in A,\andeqn \|L_n\|\ge \sigma_0
\,\,({\rm for\,\,all}\,\, n)
\eneq
such that
\beq
\lim\inf_n \|L_n(c)\|< (1/2)\|c\|\rforal c\in {\cal H}_0.
\eneq

Consider the contractive linear map  $L: A\to l^\infty(\{B_n\})$ 
defined by $L(a)=\{L_n(a)\}$ for $a\in A.$ Let $\Pi: l^\infty(\{B_n\})\to l^\infty(\{B_n\})/c_0(\{B_n\})$ be the quotient map.  Define $\phi=\Pi\circ L: A\to l^\infty(\{B_n\})/c_0(\{B_n\}).$
Then $\phi$ is a \hm. Since $\|L_n\|\ge \sigma_0,$ $\|\phi\|\ge \sigma_0>0.$
But $A$ is simple, therefore $\|\phi(a)\|=\|a\|$ for all $a\in A.$
Hence, for all large $n,$
\beq
\|L_n(a)\|\ge (3/4)\|a\|\rforal a\in {\cal H}_0.
\eneq
A contradiction.
\end{proof}

{\bf The proof of Corollary \ref{CM0}}:
%
Fix $\ep$ and ${\cal F}$ as in the corollary.
Let $\dt_1$ (as $\dt$),  ${\cal G}_1$ (as ${\cal G}$) and ${\cal H}$ be given by 
Theorem \ref{Main} for $\ep$ and ${\cal F}.$ 
Choose a finite subset  ${\cal G}_2\subset A$ (as ${\cal G}$) 
and $\eta>0$ for ${\cal H}$ given by Lemma \ref{Lsimple}.
Put ${\cal G}={\cal G}_1\cup {\cal G}_2$ and $\dt=\min\{\dt_1, \eta\}>0.$

Now suppose that $L: A\to B(H)$ is a \cpc\, with $\|\pi\circ L\|\ge \sigma$
such that \eqref{CM-1} holds. 
Then, by applying Lemma \ref{Lsimple}, 
one has 
\beq
\|\pi\circ L(a)\|\ge (1/2)\|a\|\rforal a\in {\cal H}. 
\eneq

Thus Theorem \ref{Main} applies.

\begin{rem}\label{Remarks6}
In Corollary \ref{CM0}, one  may choose $\sigma=\ep<1/4.$
However, in that case, when $A$ is a unital and non-elementary simple 
\CA,
the condition that $\|\pi\circ L\|\ge \ep$ is actually necessary.

To see this,  let us assume that $\|\pi\circ L\|<\ep<1/4$
and choose ${\cal H}\supset {\cal F}.$  Recall that we may  also assume 
that $1_A\in {\cal F}$ and $L(1_A)=e$ is a non-zero projection. 
Hence $\|L\|=1.$ 

Then, one may choose a projection $p\in {\cal K}$ such that
$\|\pi\circ L(a)(1-p)\|<\ep$ and $\|(1-p)\pi\circ L(a)\|<\ep,$ whence
\beq
\|pL(a)p-L(a)\|<2\ep\rforal  a\in {\cal F}.
\eneq
If  there were a \hm\, $h: A\to  B(H)$ such that
\beq
\|L(a)-h(a)\|<\ep\rforal a\in {\cal F}.
\eneq
Then, in particular,  
\beq
\|pL(1_A)p-h(1_A)\|<3\ep<1.
\eneq
Since $pL(1_A)p\in {\cal K}$  and $h(1_A)$ is a projection, 
this would imply that $h(1_A)\in {\cal K}.$ 
Hence $h$ is a \hm\, from $A$ into ${\cal K}.$
Since $A$ is a unital simple and non-elementary, $h=0.$ 
This is not possible as $L(1_A)=e$ and $1_A\in {\cal F}.$
\end{rem}

\section{Amenable Groups}

Let $G$ be a group. Denote by $\C[G],$ the group ring which is also the vector space (over $\C$)  of linear combinations 
of $G.$  If $G$ is  discrete, then $\C[G]=C_c(G).$ 
Denote by $C^*(G)$ the group \CA\, of $G$ and $C_r^*(G)$ the reduced \CA\, of the group 
$G.$  For the rest of this paper, we will only consider countable 
discrete amenable groups. In particular, we assume that $C^*(G)=C_r^*(G).$ 

\begin{df}\label{dtildephi}
Let $B$ be a \CA\,   and $\phi: G\to U(B)$ be a map. 
In what follows, denote by $\tilde \phi: \C[G]\to B$ the linear extension 
of $\phi.$ If $\phi$ is a group \hm, then it is   known that 
$\tilde \phi$ is a self-adjoint linear map (with the usual involution) and can be 
uniquely extended to a \CA\, \hm\, $\tilde \phi: C^*(G)\to B.$ 
\end{df}

\begin{lem}\label{Ldiscrete}
Let $G$ be a countable discrete amenable group.
For any $\ep>0,$ any finite subset
${\cal F}\subset \C[G]\subset C_r^*(G),$ 
there is $\dt=\dt(\ep, {\cal F}, G)>0,$ a finite subset ${\cal G}={\cal G}_{\ep, {\cal F}, G}\subset G$ 
satisfying the following:
if $\phi: G\to U(B)$ 
(where $B$ is a unital \CA\, with an ideal $I$)
is a map
such that 
\beq
\|\phi(gf)-\phi(g)\phi(f)\|<\dt\tforal g, f\in {\cal G}_{\ep, {\cal F}, G},
\eneq
then there exists a \cpc\, $L: C^*_r(G)\to B$
such that
\beq\label{Ldiscrete-2}
&&\|\tilde \phi(f)-L(f)\|<\ep\tforal f\in {\cal F}\tand\\
&&\|L(ab)-L(a)L(b)\|<\ep\tforal a,\, b\in {\cal F},
\eneq
where $\tilde \phi$ is the linear extension of $\phi.$ 
Moreover, suppose that   $0<\sigma<1$ is also  given.  
Then,
if 
\beq
\|q\circ \tilde \phi(c)\|\ge \sigma\|c\|\tforal c\in {\cal F},
\eneq
where $q: B\to B/I$ is a quotient map. 
we may also require that
\beq
\|q\circ L(a)\|\ge {\sigma/2}\|a\|\tforal a\in {\cal F}.
\eneq
\end{lem}

(Note that ${\cal S}$ and $\dt$ do depend on $\sigma,$
and $\|a\|$ and $\|c\|$ are the norm used in $C_r^*(G).$)

\begin{proof}

Suppose the lemma is false.
Then there are  $\ep_0>0$ and a finite subset 
${\cal F}_0\subset \C[G]$
such that the conclusion of the lemma is false.
Then there exists a sequence of maps $\phi_n: G\to U(B(H))$ such that
\beq\label{Pcpc-4}
\lim_{n\to\infty}\|\phi_n(fg)-\phi_n(f)\phi_n(g)\|=0\rforal f,g\in G,
\eneq
and, yet, 
\beq\label{Pcpc-5}
\inf \{ \max\{\|\phi_n(f)-\Lambda_n(f)\|: f\in {\cal F}_0\}\}>\ep_0,
\eneq 
where the infimum is taken among all possible \cpc s 
$\Lambda_n: C_r^*(G)\to B$ 
with\\ $\|\Lambda_n(ab)-\Lambda_n(a)\Lambda_n(b)\|<\ep_0$ for all $a, b\in {\cal F}_1.$ 

Define $\Phi: \C[G]\to  l^\infty(B)$ by $\Phi(g)=\{\tilde \phi_n(g)\}$ for all $f\in \C[G].$
Consider $\Psi:=\Pi\circ \Phi: \C[G]\to l^\infty(B)/c_0(B),$
where $\Pi: l^\infty(B) \to l^\infty(B)/c_0(B)$ is the quotient map.
Then, by \eqref{Pcpc-4},  $\Psi|_G$ is a group \hm. 
$\Psi(G).$ Since $G$ is amenable, $C^*(G)=C^*_r(G).$
Therefore there is a \hm\, $\psi: C^*_r(G)\to 
l^\infty(B)/c_0(B)$
such that $\psi|_G=\Psi.$
Since $C^*_r(G)$ is amenable, 
by the Effros-Choi lifting theorem (see \cite{CE}), there 
is a \cpc\, $L: C_r^*(G)\to l^\infty(B)$ 
such that $\Pi\circ L=\psi.$ Write $L(a)=\{L_n(a)\}$ for  all $a\in C_r^*(G).$ 
Then
\beq\label{Pdescrete-5}
\lim_{n\to\infty}\{\max\|L_n(g)-\tilde \phi_n(g)\|: g\in {\cal F}_0\}=0.
\eneq
Moreover, since $\psi$ is a \CA\, \hm, 
\beq
\lim_{n\to\infty}\|L_n(ab)-L_n(a)L_n(b)\|=0\tforal a, b\in C_r^*(G).
\eneq
These contradict with \eqref{Pcpc-5}. 
Thus the first part of the lemma holds.

For the second part of the lemma, 
let $\ep>0$ and a finite subset ${\cal F}$ be given.  
We may  assume that $0\not\in {\cal F}.$
Choose
$\eta=\min\{\ep,({\sigma\over{4}})\min\{\|g\|: g\in {\cal F}\}\}>0.$ 

Applying the first part of lemma for $\eta$ (instead of $\ep$), 
we have 
\beq
\|\tilde \phi(f)-L(f)\|<\eta\le \ep\rforal f\in {\cal F}.
\eneq
Hence, if $\|q\circ \tilde \phi(f)\|\ge \sigma\|f\|$ for all $f\in {\cal F},$
then
\beq
\|q\circ L(f)\|\ge \|\pi\circ \tilde \phi(f)\|-\eta\ge {\sigma\over{2}}\|f\|\rforal f\in {\cal F}.
\eneq
%
%
\end{proof}

\begin{thm}[Theorem \ref{Main2}]\label{CMM}
Let $G$ be a countable discrete amenable group and $H$ be an infinite dimensional Hilbert space.
Let $\ep>0$ and ${\cal F}\subset \C[G]$ be a finite subset and $1\ge \sigma>0.$
Then there exists $\dt>0,$ a finite subset ${\cal G}\subset G$    and 
a finite subset ${\cal H}\subset \C[G]$ such that, if 
$\phi: G\to U(B(H))$ is a map 
satisfying the condition that
\beq
&&\|\phi(fg)-\phi(f)\phi(g)\|<\dt\tforal f,g\in {\cal G}\tand\\\label{CMM-1}
&&\|\pi\circ \tilde \phi(a)\|\ge \sigma\|a\|\tforal a\in {\cal H},
\eneq
then there exists a \hm\, $h: G\to U(B(H))$ such that
\beq
\|\tilde \phi(f)-\tilde h(f)\|<\ep\rforal f\in {\cal F}.
\eneq
\end{thm}

\begin{proof}
Fix $\ep>0$ and a finite subset ${\cal F}$ as well as $0<\sigma\le 1.$ 
Consider the \CA\, $C_r^*(G).$ Since $G$ is    amenable,  
by \cite{Tu}, $C_r^*(G)$ satisfies the UCT and, by \cite{TWW}, is quasidiagonal.

Let $\dt_1$ (as $\dt$) and ${\cal G}_1,\, {\cal H}\subset C_r^*(G)$ (${\cal G}_1$ as ${\cal G}$) be given 
by Theorem \ref{Main} for $\ep/2$ (in place of $\ep$), ${\cal F}$ and $\sigma/2$ (as well as 
$A=C_r^*(G)$). 
Since $\C[G]$ is dense in $C_r^*(G),$ we may assume that ${\cal G}_1, \, {\cal H}\subset \C[G].$

Put $\ep_1=\min\{\ep/2, \dt_1\}>0$ and ${\cal F}_1={\cal G}_1\cup {\cal H}\cup  {\cal F}.$
Let $\dt=\dt(\ep_1, {\cal F}_1, G)$  and  ${\cal G}={\cal G}_{\ep_1, {\cal F}_1, G)}$ be given  by Lemma \ref{Ldiscrete} 
for $\sigma/2$ (in place of $\sigma$).

Now we assume that $\phi:  G\to B(H)$ is as described in the theorem for $\dt,$ ${\cal G}$ and ${\cal H}$  above.

Applying Lemma \ref{Ldiscrete}, we obtain a \cpc\, 
$L: C_r^*(G)\to B(H)$ such that
\beq\label{73m-1}
&&\|\tilde \phi(g)-L(g)\|<\ep_1\rforal g\in {\cal G}_1,\\
&&\|L(ab)-L(a)L(b)\|<\ep_1\rforal a, b\in {\cal G}_1\andeqn\\
&&\|\pi\circ L(a)\|\ge {\sigma\over{2}}\|a\|\rforal a\in {\cal H}.
\eneq
By the choice of  $\ep_1$ and $\dt_1$ as well as ${\cal G}_1,$ applying Theorem \ref{Main},
we obtain a faithful and full \hm\, 
$\Phi: C_r^*(G)\to B(H)$ such that
\beq
\|\Phi(g)-L(g)\|<\ep/2\rforal g\in {\cal F}.
\eneq
Choose $h=\Phi|_G: G\to B(H).$ We obtain  (see also \eqref{73m-1}) that 
\beq
\|\tilde h(g)-\tilde \phi(g)\|<\ep\rforal g\in {\cal F}.
\eneq
\end{proof}

\section{Counterexamples}

\begin{exm}\label{Exshift}
This is an example of maps from 
$L: C(\D)\to B(H)$  such that $\pi\circ L$ induces a \hm\,
from $C(\T)\to B(H)/{\cal K}$ which in turn induces a non-zero $K_1$-map.
Let $f\in C(\D)$ which vainshes on the unit circle. Then 
$\pi\circ L(f)=0.$ In particular, the second condition  in \eqref{Main-1}
fails. 
The example 
is well known. We present here    for our specific purpose.

Let $H$ be an infinite dimensional separable Hilbert space with 
orthonormal basis $\{e_n\}.$ Define 
\beq
s_n(e_k)=\min\{k/n, 1\}e_{k+1},\,\,\, k\in \N,
\eneq
$n=1,2,....$   One computes that
\beq
\lim_{n\to\infty}\|s_n^*s_n-s_ns_n^*\|=0.
\eneq
Let $\Pi:  \prod_{n=1}^\infty B(H)\to  \prod_{n=1}^\infty B(H)/\bigoplus_{n=1}^\infty B(H)$ be the quotient map.  Therefore $\Pi(\{s_n\})$ is a normal element with $\|\Pi(\{s_n\})\|=1.$ 
Let $\D$ be the closed unit disk and 
define  a \hm\,  $\phi: C(\D)\to \prod_{n=1}^\infty B(H)/\bigoplus_{n=1}^\infty B(H)$  by
$\phi(f)=f(\Pi(\{s_n\})$ for all $f\in C(\D).$
By Effros-Choi's lifting Theorem (\cite{CE}), there is a \cpc\, $\Phi: C(\D)\to \prod_{n=1}^\infty B(H)$ such that $\Pi\circ \Phi=\phi.$  Write $\Phi(f)=\{\Phi_n(f)\}$ for $f\in C(\D).$
Then each $\Phi_n: C(\D)\to B(H)$ is a \cpc.  Moreover, 
\beq
\lim_{n\to\infty}\|\Phi_n(fg)-\Phi_n(f)\Phi_n(g)\|=0\rforal f, g\in C(\D)
\eneq
as $\Phi$ is a \hm. However, there are  {\it no}  sequences of \hm s 
$h_n: C(\D)\to B(H)$ such that
\beq
\lim_{n\to\infty}\|h_n(\iota)-\Phi_n(\iota)\|=0,
\eneq
where $\iota: \D\to\D$ is the identity function. In fact, $s_n$ is far away 
from normal elements, as $\pi(s_n)$ is a unitary with nonzero index (see 
\cite[Example 4.6]{FR}).   

Note that $\phi_n: C(\T)\to B(H)/{\cal K}$ defined by
$\phi_n(f)=f(\pi(s_n))$ is an injective \hm\, for each $n\in \N.$
Even though   $\D$ is 
contractive,  $C(\D)$ has an interesting quotient $C(\T)$ which 
contributes to the hidden Fredholm index.  It shows the importance of the fullness condition
(the second condition in \eqref{Main-1})
in Theorem \ref{Main}.

\end{exm}

We would like to point out that Theorem \ref{CMM} does not hold without condition 
\eqref{CMM-1}. 
The following counterexample is presented here for the convenience of the reader.
It was given by Voiculescu (\cite{DV2}).  

\begin{exm}\label{ExVoi}
Let $G=\Z^2$ and $H=l^2.$ 
Let $g_1=(1,0)$ and $g_2=(0,1)$ in $\Z^2$ be the generators.
Then there exists a sequence of maps
$\phi_n : G\to U(B(H))$ such that 
\beq
\lim_{n\to\infty}\|\phi_n(fg)-\phi_n(f)\phi_n(g)\|=0\rforal f, g\in G.
\eneq
But 
\beq
\lim\inf_n\{\max\{\|\phi_n(g_i)-\psi_n(g_i)\|:  1\le i\le 2\}\}>0
\eneq
for any sequence of representations $\psi_n: G\to U(B(H)).$

This is the Voiculescu's example (\cite{DV2}) with minimum modification.
However, since our goal is different and we need an infinite dimensional 
statement (Voiculescu's statement is for finite dimensional), there will be a set of differences.  To be more precise, we will repeat most of Voiculescu's construction as well as the arguments.  Moreover, we will try to use the same, or similar notations.

We first fix an integer $n\in \N$ ($n\ge 10$ as indicated in Voiculescu's example).

Let $\{e_k\}$ be an orthonormal basis for $H.$ 
 Define two unitaries 
 \beq
 U_n(e_k)&=&e_{k+1}, \,\,1\le k\le n-1,\,\, U_n(e_n)=e_1, U_n(e_k)=e_k,\,\,\, k>n,\andeqn\\
 V_n(e_k)&=&\exp(2k\pi \sqrt{-1}  /(n+1)) e_k,\,\,\, 1\le k\le n, V_n(e_k)=e_k,\,\,\, k>n.
 \eneq
 Then $U_n, \, V_n\in U(B(H)),$ $n\in \N.$ 
Let  $Q_n$  be the projection on the span 
of $\{e_1, e_2,...,e_n\}.$
Then $Q_nU_n=U_nQ_n$ and $V_nQ_n=Q_nV_n.$
Define $U_{n,0}=U_nQ_n$ and $V_{n,0}=V_nQ_n.$
Note that $U_{n,0}$ and $V_{n,0}$ are unitaries on $B(Q_nH).$
It should also be noted that $1\not\in {\rm sp}(V_{n,0}).$ 

We first  compute that
\beq
\|U_nV_n-V_nU_n\|\le |1-\exp(-2\pi \sqrt{-1}/(n+1))|\to 0\,\,\, {\rm as}\,\,n\to\infty.    
\eneq

Assuming there are  pairs of commuting unitaries $U_n'$ and $V_n'$ 
in $B(H)$ such that
\beq
\lim_{n\to\infty}\|U_n-U_n'\|=0\andeqn \lim_{n\to\infty}\|V_n-V_n'\|=0,
\eneq
we will reach a contradiction.

Following Voiculescu's notation, 
consider the unit circle $\T=\{z\in \C: |z|=1\}$ 
 and the arcs
\beq
&&\Gamma=\{z\in \T: \pi/5\le \arg(z)< 4\pi/5\},\,\, \Gamma'=\{z\in \T: 2\pi/5\le \arg(z)< 3\pi/5\},\\
&&\Gamma''=\{z\in \T: 0< \arg(z)<\pi\},\,\, \Phi':=\{z\in \T: 0<\arg(z)< 2\pi/5\},\\
&&\Phi''=\{3\pi/5\le \arg(z)<\pi\}.
\eneq

Let $E_n$ be the spectral projection of $V_n'$ corresponding to $\Gamma,$
let 
$E_n'$ be the spectral projection of $V_n$ corresponding to $\Gamma',$ 
$E_n''$ be the spectral projection of $V_n$ corresponding to 
$\Gamma'',$ $F_n^{(1)}$ to $\Phi'$ and $F_n^{(2)}$ to $\Phi'',$ 
respectively. 
Note that 
\beq\label{dv-1}
E_n''=E_n'+F_n^{(1)}+F_n^{(2)}.
\eneq
Also, since $[U_n',V_n']=0,$  
we have $[U_n', E_n]=0$ and 
hence
\beq\label{V1}
\lim_{n\to\infty}\|[U_n, E_n]\|=0.
\eneq
As in Voicuelscu's argument, we will use the following fact:
If $N_n, N_n'$ are normal operators, $\lim_{n\to\infty}\|N_n-N_n'\|=0,$ 
and $\|N_n\|$ is (uniformly) bounded, and $P_n,$ $P_n'$ are spectral 
projections of $N_n$ and $N_n',$ respectively, corresponding to Borel sets $\Omega,$
$\Omega'$ such that $\Omega\cap \Omega'=\emptyset,$ then 
$\lim_{n\to\infty}\|P_nP_n'\|=0.$  
Moreover, this fact also implies that, if $\Omega\supset \Omega',$
then 
\beq
\lim_{n\to\infty}\|P_nP_n'-P_n'\|=0=\lim_{n\to\infty}\|P_n'P_n-P_n'\|=0.
\eneq 
In particular, this gives
\beq\label{V2}
\lim_{n\to\infty}\|(1-E_n'')E_n\|=\lim_{n\to\infty}\|(I-E_n)E_n'\|=0.
\eneq
Or, equivalently,
\beq\label{V2++}
\lim_{n\to\infty}\|E_n-E_n''E_n\|=0\andeqn \lim_{n\to\infty}\|E_n'-E_nE_n'\|=0.
\eneq
Let $F_n^{(2')}$ and $F_n^{(2'')}$ be the spectral projections of $V_n$ corresponding 
to $\{z\in \T:3\pi/5\le \arg(z)<4\pi/5\}$ and $\{z\in \T: 4\pi/5\le\arg(z)<\pi\},$
respectively.  Note that $F_n^{(2)}=F_n^{(2')}+F_n^{(2'')}.$ 
Then 
\beq
\lim_{n\to\infty}\|E_n F_n^{(2'')}\|=0,\,\,\lim_{n\to\infty}\|E_nF_n^{(2')}-F_n^{(2')}\|=0
\andeqn \lim_{n\to\infty}\|F_n^{(2')}E_n-F_n^{(2')}\|=0.
\eneq
It follows that
\beq
\lim_{n\to\infty}\|E_nF_n^{(2)}-F_n^{(2)}E_n\|=0.
\eneq
Similarly, 
\beq
\lim_{n\to\infty}\|E_nF_n^{(1)}-F_n^{(1)}E_n\|=0.
\eneq
Moreover,
\beq\label{V3}
\lim_{n\to\infty}\|F_n^{(1)} E_n F_n^{(2)}\|=\lim_{n\to\infty}\|F_n^{(1)}E_n F_n^{(2')}\|
=\lim_{n\to\infty}\|E_n F_n^{(1)}F_n^{(2')}\|=0.
\eneq
Then
\beq\label{V3+1}
&&E_nE_n''=(E_nE_n'+E_nF_n^{(1)}+E_nF_n^{(2)}),\\\label{V3+2}
 &&\lim_{n\to\infty}\|E_nF_n^{(1)}-F_n^{(1)}E_nF_n^{(1)}\|=0,\\\label{V3+3}
 &&\lim_{n\to\infty}\|E_nF_n^{(2)}-F_n^{(2)}E_nF_n^{(2)}\|=0.
\eneq
Let $X_n=E_n'+F_n^{(1)}E_nF_n^{(1)}+F_n^{(2)} E_nF_n^{(2)}.$
Then, by \eqref{dv-1}, \eqref{V2++}, \eqref{V3+1}, \eqref{V3+2} and \eqref{V3+3},
\beq\label{V3+}
\lim_{n\to\infty}\|X_n-E_n\|=0.
\eneq
Hence $\lim_{n\to\infty}\|X_n^2-X_n\|=0.$  Define 
\beq\label{V4-}
\tilde E_n=f(X_n)\,\,\,{\rm for\,\,\, large\,\,} n,
\eneq
where $f\in C_0((0, \infty))_+$ such that 
$0\le f\le 1,$ $f(t)=0$ for $t\in [0, 1/4]$ and $f(t)=1$
for $t\in [1/2, 1].$ Note that (for all large $n$) $\tilde E_n$ is the spectral projection 
of $X_n$ corresponding to $[1/2, 1].$ 
Note that
\beq
E_n'\le X_n\le E_n'+F_n^{(1)}+F_n^{(2)}=E_n''.
\eneq
Hence
\beq\label{V4}
E_n'\le \tilde E_n\le E_n''.
\eneq 
Note that $\tilde E_n=E_n'+\tilde F_n^{(1)}+\tilde F_n^{(2)},$
where $\tilde F_n^{(1)}\le F_n^{(1)}$ and $\tilde F_n^{(2)}\le F_n^{(2)}$ are projections.

Consider now the projection $E_n^+=E_n'+F_n^{(1)}+\tilde F_n^{(2)}$ and assume, 
as in \cite{DV2}, $n\ge 10.$ We have (see \eqref{dv-1} ) that
\beq\label{V5}
&&\tilde E_n\le E_n^+\le E_n''.
\eneq
Recall that  $Q_j$ is the projection on the span of $\{e_1,e_2,...,e_j\},$ $j\in \N.$ 
Then  (from the definition of  $V_n$) $F_n^{(1)}\le Q_{J_n}$ for some $
(n+1)/5\le J_n\le ((n+1)/5)+1$ and 
$U_nF_n^{(1)}=Q_{J+1}U_nF_n^{(1)}.$ When $n\ge 20,$ $Q_{J+1}\le F_n^{(1)}+E_n'.$ 
Hence we also have
\beq\label{V6}
0=
(I-E_n^+)U_nF_n^{(1)}=(I-E_n^+)U_n\tilde F_n^{(1)}.
\eneq
It follows that (see \eqref{V5} for the third equality)
\beq\label{V7}
(I-E_n^+)U_nE_n^+=(I-E_n^+)U_n(E_n'+\tilde F_n^{(2)})=(1-E_n^+)U_n\tilde E_n\\
=(1-E_n^+)(I-\tilde E_n)U_n\tilde E_n.
\eneq
Also, by \eqref{V1}, \eqref{V3+} and \eqref{V4-}, 
\beq\label{V8}
\lim_{n\to\infty}\|(I-\tilde E_n)U_n\tilde E_n\|=0.
\eneq
which (together with \eqref{V7}) implies that
\beq\label{V9}
\lim_{n\to\infty}\|(I-E_n^+)U_nE_n^+\|=0.
\eneq

Now consider the shift operator $S$ which is given 
by $S(e_k)=e_{k+1}$ for all $k\in \N.$
Then 
\beq\label{V10}
SQ_{n-1}=U_nQ_{n-1}.
\eneq
Recall also
$E_n^+\le E_n'' \le Q_{\lfloor(n+1)/2\rfloor +1},$  where $\lfloor (n+1)/2\rfloor$ is the integer part of 
$(n+1)/2.$
It follows 
that
\beq
SE_n^+=U_nE_n^+.
\eneq
We then  compute that
\beq
(I-E_n^+)U_nE_n^+=(I-E_n^+)SE_n^+. 
\eneq
It follows from \eqref{V9} that
\beq\label{V12}
\lim_{n\to\infty}\|(I-E_n^+)SE_n^+\|=0.
\eneq
Recall that $E_n^+\le Q_{\lfloor (n+1)/2\rfloor+1}.$ So $E_n^+$ has finite rank.
For each $k\in \N,$ if $n\ge \max\{5(k+1), 10\},$ 
then 
\beq
F_{n}^{(1)}\ge
%
Q_k.
\eneq
In other words, $\{F_n^{(1)}\}$ forms an approximate identity for ${\cal K}.$
Therefore \eqref{V12} would imply that $S$ is quasidiagonal which is not.
We reach a contraction. 
\end{exm}

\begin{rem}
Note that, in the example above, 
$\pi\circ \phi_n(g_1)=\pi\circ \phi_n(g_2)=1_{B(H)}.$ 
However, 
$\tilde \phi_n$ is approximately injective (on $\C[G]$).  
This is an extremal case that condition \eqref{MM-2}  in 
Theorem \ref{Main2} fails.  If we choose $A=C^*(\Z^2),$
this example is also a counterexample of {\bf Q1} and 
an extremal case that the second condition in \eqref{Main-1} fails
(and quite different from Example \ref{Exshift}).
\end{rem}

 \providecommand{\href}[2]{#2}


\noindent 
hlin@uoregon.edu,\\
%
hlin@simis.cn

\end{document}